\documentclass[english]{amsart}

\usepackage{amsmath,amssymb} 
\usepackage{mathrsfs} 
\usepackage[usenames,dvipsnames]{xcolor} 
\usepackage{bbm} 
\usepackage[scr=boondoxo]{mathalfa} 
\usepackage{dsfont} 
\usepackage[
	colorlinks,
	plainpages,
	citecolor=red!50!black,
    filecolor=Darkgreen,
    linkcolor=blue!40!black,
    urlcolor=cyan!50!black!90]{hyperref} 

\usepackage{version} 

\usepackage{babel}  
\usepackage{csquotes}
\MakeOuterQuote{"}
\usepackage[all]{xy} 
\usepackage{accents} 
\usepackage{enumitem} 

\newtheorem{theorem}{Theorem}[section]
\newtheorem{proposition}[theorem]{Proposition}
\newtheorem{corollary}[theorem]{Corollary}
\newtheorem{lemma}[theorem]{Lemma}
\newtheorem{definition}[theorem]{Definition}
\newtheorem{remark}[theorem]{Remark}

\newcommand{\wt}{\widetilde}
\newcommand{\wh}{\widehat} 

\newcommand{\ot}{\otimes}

\newcommand{\gr}{\mathrm{gr}\,}
\newcommand{\End}{\mathrm{End}}

\newcommand{\onto}{\twoheadrightarrow}

\makeatletter  
\newcommand{\sbullet}{%
  \hbox{\fontfamily{lmr}\fontsize{.4\dimexpr(\f@size pt)}{0}\selectfont\textbullet}}
\DeclareRobustCommand{\mathbullet}{\accentset{\sbullet}}
\makeatother

\newcommand{\mfa}{\mathfrak{a}}

\newcommand{\mfg}{\mathfrak{g}}

\newcommand{\mfH}{\mathfrak{H}}

\newcommand{\mfs}{\mathfrak{s}}

\newcommand{\mft}{\mathfrak{t}}

\newcommand{\mfgl}{\mathfrak{g}\mathfrak{l}}
\newcommand{\mfsl}{\mathfrak{s}\mathfrak{l}}

\newcommand{\mcA}{\mathcal{A}}

\newcommand{\mcC}{\mathcal{C}}
\newcommand{\mcD}{\mathcal{D}}
\newcommand{\mcE}{\mathcal{E}}
\newcommand{\mcF}{\mathcal{F}}

\newcommand{\mcH}{\mathcal{H}}
\newcommand{\mcI}{\mathcal{I}}

\newcommand{\mcP}{\mathcal{P}}

\newcommand{\mcV}{\mathcal{V}}

\newcommand{\mcX}{\mathcal{X}}

\newcommand{\mbA}{\mathbf{A}}

\newcommand{\mbC}{\mathbf{C}}

\newcommand{\mbD}{\mathbf{D}}

\newcommand{\mbF}{\mathbf{F}}

\newcommand{\mbG}{\mathbf{G}}

\newcommand{\mbH}{\mathbf{H}}

\newcommand{\mbU}{\mathbf{U}}

\newcommand{\mbX}{\mathbf{X}}

\newcommand{\C}{\mathbb{C}}

\newcommand{\Z}{\mathbb{Z}}

\newcommand{\mdH}{\mathds{H}}

\newcommand{\mdX}{\mathds{X}}

\newcommand{\msc}{\mathsf{c}}

\newcommand{\al}{\alpha}
\newcommand{\del}{\delta}
\newcommand{\eps}{\epsilon}
\newcommand{\veps}{\varepsilon}

\newenvironment{NB}{
\color{blue}{\bf Note}. \footnotesize
}{}
\excludeversion{NB}  

\title{Vertex Representations for Yangians of Kac-Moody algebras}
\author[N. Guay]{Nicolas Guay}
\address{N.G.: University of Alberta, Department of Mathematical and Statistical Sciences, Edmonton, AB T6G 2G1, Canada.}
\email{nguay@ualberta.ca}
\author[V. Regelskis]{Vidas Regelskis}
\address{V.R.: University of York, Department of Mathematics, York, YO10 5DD, UK and
Vilnius University, Institute of Theoretical Physics and Astronomy, Saul\.etekio av.~3, Vilnius 10257, Lithuania.}
\email{vidas.regelskis@gmail.com}
\author[C. Wendlandt]{Curtis Wendlandt}
\address{C.W.: University of Alberta, Department of Mathematical and Statistical Sciences, Edmonton, AB T6G 2G1, Canada.}
\email{cwendlan@ualberta.ca}


\numberwithin{equation}{section}

\begin{document}
\begin{abstract}
Using vertex operators, we build representations of the Yangian of a simply laced Kac-Moody algebra and of its double. As a corollary, we prove the PBW property for simply laced affine Yangians.
\end{abstract}

\maketitle

\setcounter{tocdepth}{1} 
\tableofcontents

\allowdisplaybreaks


\section{Introduction}

Vertex operators originate from dual resonance models in theoretical physics. They were used by I. Frenkel and V. Kac in their groundbreaking paper \cite{FrKa} to build an explicit realization of the basic representation of a simply-laced affine Lie algebra. Their work was later extended to non-simply laced affine Lie algebras \cite{BTM,GNOS}, to quantum affine algebras \cite{FrJi,Be,JiMi,JKM,Ji3,Ji4,ChJi}, to twisted quantum affine algebras and more general quantum Kac-Moody algebras \cite{Ji1,Ji2}, to toroidal and quantum toroidal algebras \cite{MRY,Sa}, and to Lie superalgebra (e.g. \cite{KSU}). 

In this paper, we address the problem of developing an analogue of the work of I. Frenkel and V. Kac for Yangians of simply laced Kac-Moody algebras. Yangians form an important family of quantum algebras which originate from physics, but were first properly defined in general by V. Drinfeld in \cite{Dr}.  They can be obtained from quantum loop algebras via a limit procedure \cite{Gu} and it turns out that Yangians and quantum loop algebras become isomorphic after passing to certain completions \cite{GTL1}. The first goal of this paper is to construct representations of Yangians, via their centrally extended doubles (see Definition \ref{DY:def}), using vertex operators which act on a tensor product of a Fock space with a twisted group algebra (see Theorem \ref{thm:Vrep}). In the case of the Yangian associated to $\mfsl_n$ and $\mfgl_n$, this was done in \cite{IoKo,Io,Kh}. It should be noted that our construction is not a direct consequence of the work of I. Frenkel and N. Jing \cite{FrJi,Ji2} on vertex operator representations of quantum affinizations associated to symmetric Kac-Moody algebras. Indeed, our construction differs in at least one essential way from the one in \cite{FrJi,Ji2}, namely that we use a different lattice to build the underlying Fock space. 

The second goal of this paper is to prove a version of the Poincar\'{e}-Birkhoff-Witt Theorem for affine Yangians of simply laced type (Theorem \ref{T:PBW}) using the vertex representations of Theorem \ref{thm:Vrep}. For Yangians associated to simple Lie algebras, this theorem was proved in general in \cite{Le}, and for classical Lie algebras, a version of the PBW Theorem stated in terms of the RTT-presentation of the Yangian can be found in \cite{Mo} and \cite{AMR}; for affine Yangians, only the type $A^{(1)}$ has been considered before \cite{Gu}.   A separate proof of the PBW property for simply laced affine Yangians has been announced in \cite{YaZh2}. The argument in \textit{loc. cit.}, which is of independent interest, uses the existence of a morphism from the Yangian of $\mfg$ to the reduced Drinfeld double of the spherical subalgebra of a shuffle algebra associated to $\mfg$ \cite[Cor.\ 3.4]{YaZh1}.

Our paper is structured as follows. Section \ref{Sec:Yg} presents the definition of the Yangian $Y(\mfg)$ associated to a symmetrizable Kac-Moody algebra $\mfg$ and describes its classical limit as the enveloping algebra of a certain Lie algebra $\mfs$ (Proposition \ref{P:s->Yh}) which coincides with the current algebra $\mfg[t]$ when $\mfg$ is finite-dimensional. We also recall results of Section 6 from \cite{GNW} about a certain parameter dependent coproduct (Theorem \ref{T:coprodu}) which will be needed in Subsection \ref{Ss:proofmain} in order to build a faithful representation of $Y(\mfg)$. It is possible to repeatedly apply this coproduct, but since it is not coassociative, one should proceed with care, as explained at the end of Section \ref{Sec:Yg}.

In Section \ref{Sec:DYg}, we give the definition of the centrally extended Yangian double $DY^{\msc}(\mfg)$ of $\mfg$ and study its basic properties. When $\mfg$ is a finite-dimensional simple Lie algebra, a definition of $DY^\msc(\mfg)$ was given over twenty year ago in \cite{Kh}, where it was conjecturally described as the Hopf algebra double of a central extension of $Y(\mfg)$. Although this interpretation seems to be limited to that setting, a general definition can be obtained by inserting an arbitrary Cartan matrix into the explicit definition of $DY^{\msc}(\mfg)$ provided in \cite{Kh,DiKh}. This procedure leads to Definition \ref{DY:def}. 

After giving the definition of $DY^\msc(\mfg)$ (Definition \ref{DY:def}, Lemma \ref{DY':def}), we relate its classical limit to the enveloping algebra of a certain Lie algebra $\mft$ (Proposition \ref{P:mftiso}), which in the finite-dimensional setting is just the affine Lie algebra $\mfg[t^{\pm 1}]\oplus \C K$ associated to $\mfg$. We conclude Section \ref{Sec:DYg} with Proposition \ref{P:iota}, which makes precise how the Yangian maps into the centrally extended Yangian double. 

The aforementioned Lie algebras $\mfs$ and $\mft$ can also be described more explicitly when $\mfg$ is an untwisted affine Lie algebra: in this case, they are isomorphic to the universal central extensions of two loop algebras. This fact was proved in \cite{MRY} and Section \ref{Sec:aff} serves to recall this description. In Proposition \ref{P:g'[t]}, we show that $\mfs$ and $\mft$ can be equivalently characterized as the universal central extensions of $\mfg^\prime[t]$ and $\mfg^\prime[t^{\pm 1}]$, respectively, where $\mfg^\prime=[\mfg,\mfg]$ is the derived subalgebra of $\mfg$. This description of $\mfs$ and $\mft$ is also valid when $\mfg$ is finite-dimensional. Our PBW Theorem for $Y(\mfg)$ (namely, Theorem \ref{T:PBW}) is stated as providing an isomorphism between the associated graded ring of $Y(\mfg)$ (for a certain filtration) and the enveloping algebra of $\mfs$, so the results of Section \ref{Sec:aff} are relevant for our second main theorem.

The main section of this paper is Section \ref{Sec:Ver}. Assuming that $\mfg$ is a simply laced Kac-Moody algebra, we construct a representation of the Yangian double $DY^{\msc}(\mfg)$ (and thus of the Yangian $Y(\mfg)$) which is given by vertex operators and which factors through the Yangian double at level one (see Theorem \ref{thm:Vrep} and also Proposition \ref{P:Vexp} and Corollary \ref{C:spec} for slightly different versions of that theorem). This representation can be realized in a space built from the tensor product of a Fock space with the twisted group algebra $\C_\veps[Q]$ of the root lattice $Q$: see Definition \ref{def:Ce[Q]} and \eqref{V}. Its construction generalizes, and has been motivated by, the results of Iohara \cite{Io} for $\mfg=\mfsl_N$, as well as the results of \cite{Kh} and \cite{IoKo} which were stated for $\mfg=\mfsl_2$ and $\mfg=\mfgl_2$, respectively. By considering carefully a certain filtration, our construction leads to a representation of the Lie algebra $\mft$ (Corollary \ref{C:rho_0}) which is related, but not always isomorphic, to the representation of $\mft$ obtained from the classical vertex representation construction \cite{FrKa,MRY}: this is made precise in Proposition \ref{P:V_A->V}. 

The last section contains a proof of the PBW Theorem for affine, simply laced Yangians: see Theorem \ref{T:PBW}. We prove that the associated graded ring of the Yangian $Y(\mfg)$ (for a certain filtration) is isomorphic to the enveloping algebra $U(\mfs)$ of $\mfs$. 
As a consequence, we obtain in Theorem \ref{T:flat} that the $\C[\hbar]$-algebra version of the Yangian $Y_\hbar(\mfg)$ (see Definition \ref{D:Y(g)}) is a flat deformation of the enveloping algebra $U(\mfs)$ of $\mfs$. The main point of the proof of Theorem \ref{T:PBW} is to show the injectivity of the natural epimorphism from $U(\mfs)$ to the associated graded ring given in Proposition \ref{P:weakPBW}: this is accomplished by taking tensor products of the vertex representation of $Y(\mfg)$ constructed in Section \ref{Sec:Ver} (actually, it is necessary to consider a slightly larger Kac-Moody algebra) and, by using a carefully chosen filtration, reducing the proof to the question of the faithfulness of the corresponding vertex representation of $\mfs$, which was addressed previously in \cite{MRY}. In Appendix \ref{app:A}, we prove that the collection of tensor powers of any faithful representation for an arbitrary complex Lie algebra separate points of its enveloping algebra. This will be applied to the Lie algebra $\mathfrak{s}$ to prove Theorem \ref{T:PBW}.

\subsection*{Acknowledgements} The authors thank Yaping Yang and Gufang Zhao for sharing a preliminary version of their proof of the Poincar\'{e}-Birkhoff-Witt Theorem for simply-laced affine Yangians using the shuffle algebra approach. They are very thankful to one referee for a very careful reading of their manuscript. The first and third named authors gratefully acknowledge the financial support of the Natural Sciences and Engineering Research Council of Canada provided via the Discovery Grant Program and the Alexander Graham Bell Canada Graduate Scholarships - Doctoral Program, respectively. The second named author acknowledges the financial support of the European Social Fund via grant number 09.3.3-LMT-K-712-02-0017.

\section{The Yangian of \texorpdfstring{$\mfg$}{g}}\label{Sec:Yg}

In this section we recall the definition of the Yangian and give some of its basic properties.
Let $\mfg$ be a symmetrizable Kac-Moody algebra associated to an indecomposable Cartan matrix $\mbA=(a_{ij})_{i,j\in I}$, where $I$ is an indexing set for the simple roots of $\mathfrak{g}$. We assume that $\mbA$ satisfies the condition  
\begin{equation}\label{A-cond}
 \min\{|a_{ij}|,|a_{ji}|\}\leq 1 \quad \forall \; i,j\in I\; \text{ with }\; i\neq j. 
\end{equation}
Though the constraint given by \eqref{A-cond} will not play a role until Section \ref{Sec:aff}, the results of \cite{Ji2,Na}, together with Lemma \ref{L:min} and  Remark \ref{R:arb}, suggest that the definition of 
the Yangian (and its centrally extended double) must be modified in order to extend the vertex representation construction of Section \ref{Sec:Ver} beyond the simply-laced case. 
Let $(\,,\,)$ be a fixed non-degenerate invariant symmetric bilinear form on $\mfg$. We denote by $\{\alpha_i\}_{i\in I}$ the set of simple positive roots. Set  \[d_{ij}=\tfrac{1}{2}(\al_i,\al_j)\quad \forall \; i,j\in I.\]

\subsection{Definition of the Yangian}\label{ssec:defY}
\begin{definition}\label{D:Y(g)}
The Yangian $Y_\hbar(\mfg)$ is the unital associative $\C[\hbar]$-algebra generated by the elements $x^\pm_{ir}$, $h_{ir}$, for $i\in I$ and $r\in\Z_{\geq0}$, subject to the relations
\begin{gather}
  [h_{i r}, h_{j s}] = 0, \label{eq:relHH}\\
  [h_{i 0}, x^\pm_{j s} ] = \pm 2d_{ij} x^\pm_{j s},
  \label{eq:relHX}
\\
  [x^+_{i r}, x_{j s}^-] = \delta_{ij} h_{i, r+s}, \label{eq:relXX}
\\
  [h_{i, r+1}, x^\pm_{j s}] - [h_{i r}, x^\pm_{j, s+1}] =
  \pm\hbar d_{ij}\left(
    h_{i r} x^\pm_{j s} + x^\pm_{j s} h_{i r}
  \right), \label{eq:relexHX}
\\
  [x^\pm_{i, r+1}, x^\pm_{j s}] - [x^\pm_{i r}, x^\pm_{j, s+1}] = 
  \pm\hbar d_{ij} \left(
    x^\pm_{i r} x^\pm_{j s} + x^\pm_{j s} x^\pm_{i r}\right),
   \label{eq:relexXX}
\\
  \sum_{\sigma\in S_{m}}
   [x^\pm_{i r_{\sigma(1)}}, [ x^\pm_{i r_{\sigma(2)}}, \cdots,
       [x^\pm_{i, r_{\sigma(m)}}, x^\pm_{j s}] \cdots ]]
   = 0
   \quad \text{for $i\neq j$ and  $m=1-a_{ij}$}.
 \label{eq:relDS}
\end{gather}
In the last relation, $S_m$ denotes the symmetric group.
\end{definition}
\begin{remark}
 In the notation of \cite{GNW}, the above algebra is equal to $Y_\hbar(\mfg^\prime)$, where $\mfg^\prime$ is the derived subalgebra $[\mfg,\mfg]$. For the definition of the full Yangian, see  Definition 2.1 of \cite{GNW}. For $\mbA$ of finite or affine type, the condition \eqref{A-cond} only excludes type $A_1^{(1)}$. In the latter case, the appropriate definition of the Yangian is given in \cite[\S 1.2]{BeTs} and \cite[Def.\ 5.1]{Ko}. 
\end{remark}
Note that $Y_\hbar(\mfg)$ is generated, as a $\C[\hbar]$-algebra, by $x^\pm_{ir}$, $h_{ir}$, for $i\in I$ and $0\leq r\leq 1$: see \cite[(2.10)]{GNW}.
We also observe that $Y_\hbar(\mfg)$ is equipped with a $\Z_{\geq 0}$-grading determined by 
\begin{equation*}
\deg \hbar=1  \quad \text{ and }\quad \deg x_{ir}^\pm=\deg h_{ir}=r \quad \forall \; i\in I,r\geq 0.
\end{equation*}

We now give an equivalent definition of $Y_\hbar(\mfg)$ in terms of generating series which will prove useful in Section \ref{Sec:DYg}. The following result is a translation of \cite[Prop.\ 2.3]{GTL2}.  
\begin{proposition}[Prop.\ 2.3 of \cite{GTL2}]\label{P:Y-op}
Let $x_i^\pm(z)=\sum_{r\geq 0}x_{ir}^\pm z^{-r-1}$ and $h_i(z)=\sum_{r\geq 0}h_{ir}z^{-r-1}$ for each $i\in I$. The defining relations of $Y_\hbar(\mfg)$ are equivalent to
\begin{gather}
 h_i(z)h_j(w)=h_j(w)h_i(z), \label{Y:hh}\\
 (z-w\mp \hbar d_{ij})h_i(z)x_j^\pm(w)=(z-w\pm \hbar d_{ij})x_j^\pm(w)h_i(z)\pm 2d_{ij} x_j^\pm (w)-[h_i(z),x_{j0}^\pm], \label{Y:xh}\\
 (z-w\mp \hbar d_{ij})x_i^\pm(z)x_j^\pm(w)=(z-w\pm \hbar d_{ij})x_j^\pm(w)x_i^\pm(z)+[x_{i0}^\pm,x_j^\pm(w)]-[x_i^\pm(z),x_{j0}^\pm],  \label{Y:xx}\\
 (z-w)[x_i^+(z),x_j^-(w)]=\delta_{ij}(h_i(w)-h_i(z)), \label{Y:xxh}\\
 \sum_{\sigma \in S_{m}} [x_i^{\pm}(z_{\sigma(1)}), [x_i^{\pm}(z_{\sigma(2)}), \cdots, [x_i^{\pm}(z_{\sigma(m)}),x_j^{\pm}(w)] \cdots]] = 0, \label{Y:serre}
\end{gather}
where in the last relation $i\neq j$ and $m=1-a_{ij}$.
\end{proposition}
Multiplying the relation \eqref{Y:xxh} by $z^{-1}$ and taking the residue at $z=0$ yields 
\begin{equation}\label{Y:xxh0}
 [x_{i0}^+,x_j^-(w)]=\delta_{ij} h_i(w) \quad \forall \; i,j\in I. 
\end{equation}
Conversely, we have the following: 
\begin{proposition}[Prop.\ 3.3(3) of \cite{AG}]\label{P:Yxxh}
 The relation \eqref{Y:xxh} is a consequence of \eqref{Y:hh}, \eqref{Y:xh} and \eqref{Y:xxh0}. 
\end{proposition}
\begin{NB}
\begin{proof} We include a brief proof different from that given in \cite{AG}. 

 Multiplying \eqref{Y:xh} by $z^{-1}$ and taking the residue at $z=0$ gives 
 \begin{equation}\label{[h_i,x_j(w)]}
  [h_{i0},x_j^\pm(w)]=\pm 2 d_{ij} x_j^\pm (w). 
 \end{equation}
 Multiplying \eqref{Y:xh} instead by $w^{-2}$ and taking the residue at $w=0$ gives 
 \begin{align*}
  [h_{i1},x_j^\pm(w)]&=w[h_{i0},x_j^\pm(w)]\pm \hbar d_{ij}(h_{i0}x_j^\pm(w)+x_j^\pm(w)h_{i0})-[h_{i0},x_{j0}^\pm]\\
                     &=\pm 2d_{ij}(wx_j^\pm (w)-x_{j0}^\pm)+\tfrac{\hbar}{2}[h_{i0}^2,x_j^\pm(w)], 
 \end{align*}
 where we have employed \eqref{[h_i,x_j(w)]} to obtain the second equality. We therefore have
 \begin{equation}
  [\tilde h_{i1},x_j^\pm(w)]=\pm 2d_{ij}(wx_j^\pm (w)-x_{j0}^\pm) \quad \forall \; i,j\in I, \label{[h_1,x_j(w)]}
 \end{equation}
 where $\tilde h_{i1}=h_{i1}-\tfrac{\hbar}{2}h_{i0}^2$. Arguing by induction on $r\geq 0$, with \eqref{[h_i,x_j(w)]} as the base case, we obtain 
 \begin{equation*}
  [x_{ir}^+,x_j^-(w)]=\delta_{ij}\left(w^r h_i(w)-\sum_{k=0}^{r-1} w^{r-1-k}h_{ik}\right)\quad \forall \; i,j\in I,
 \end{equation*}
 where the inductive step is proven by applying $\mathrm{ad}(\tilde h_{i1})$ and using \eqref{Y:hh}. Multiplying this identity by $(z-w)z^{-r-1}$ and taking the sum over $r\geq 0$ gives \eqref{Y:xxh}. 
\end{proof}
\end{NB}

For each $\zeta\in \C$, let $Y_\zeta(\mfg)$ be the $\C$-algebra generated by $\{x_{ir}^\pm, h_{ir}\}_{i\in I, r\geq 0}$ subject to the defining relations 
of Definition \ref{D:Y(g)} with $\hbar$ replaced by  $\zeta$. Equivalently, 
\begin{equation*}
Y_\zeta(\mfg)=Y_\hbar(\mfg)/(\hbar-\zeta)Y_\hbar(\mfg).
\end{equation*}
For the remainder of this paper our focus will primarily be on the Yangian $Y(\mfg)=Y_1(\mfg)$. The emphasis on the single choice $\zeta=1$ is justified by the fact that the assignment 
\begin{equation}
 x_{ir}^\pm, h_{ir} \in Y(\mfg) \mapsto \zeta^{-r}x_{ir}^\pm, \zeta^{-r}h_{ir} \in Y_\zeta(\mfg) \label{zeta-iso}
\end{equation}
extends to an isomorphism of algebras $Y(\mfg)\to Y_\zeta(\mfg)$ for each fixed $\zeta\in \C^\times$. 
Note that $Y(\mfg)$ is no longer a $\Z_{\geq 0}$-graded algebra, but rather a $\Z_{\ge 0}$-filtered algebra with ascending filtration $\{\mbF_k\}_{k\geq 0}$ determined by assigning filtration degrees $r$ to $x_{ir}^\pm$ and $h_{ir}$ for each $i\in I$ and $r\geq 0$.  

\subsection{The classical limit}\label{ssec:clY}

\begin{definition}\label{D:s}
Let $\mfs$ be the Lie algebra generated by $\{X_{ir}^\pm, H_{ir}\}_{i\in I,r\geq 0}$ subject to the defining relations
\begin{gather}
  [H_{ir},H_{js}]=0, \label{s-HH}\\
  [H_{ir},X_{js}^\pm]=\pm 2d_{ij}X_{j,r+s}^\pm, \label{s-HX}\\
  [X_{ir}^+,X_{js}^-]=\delta_{ij}H_{i,r+s}, \label{s-X+X-}\\
  [X_{i,r+1}^\pm,X_{js}^\pm]=[X_{ir}^\pm,X_{j,s+1}^\pm], \label{s-xx}\\
  \mathrm{ad} (X_{i0}^\pm)^{1-a_{ij}}(X_{jr}^\pm)=0 \quad \text{ for }\; i\neq j. \label{s-serre}
\end{gather}
%
\end{definition}
Note that $\mfs$ is a $\Z_{\geq 0}$-graded Lie algebra with $\deg X_{ir}^\pm=\deg H_{ir}=r$ for all $i\in I$ and $r\geq 0$. 

In addition, $\mfs$ is always an extension of the current algebra $\mfg^\prime[t]$. Indeed, if $\{x_i^\pm,h_i\}_{i\in I}$ denote the Chevalley generators of $\mfg^\prime$, normalized so that $(x_i^+,x_i^-)=1$ and $h_i=[x_i^+,x_i^-]$, then the assignment 
\begin{equation}\label{s->g[t]}
 X_{ir}^\pm \mapsto x_i^\pm\ot t^r,\; H_{ir} \mapsto h_i\ot t^r \quad \forall \; i\in I \;\text{ and }\; r\geq 0
\end{equation}
determines a surjective Lie algebra morphism $\mfs\onto \mfg^\prime[t]$. This is an isomorphism when $\mfg$ is finite-dimensional, which can be proved using the arguments in Section 3 of \cite{MRY}, but in general this is not the case. We will consider the situation where $\mfg$ is of affine type in more detail in Section \ref{Sec:aff}. 

The next proposition illustrates that $Y_\hbar(\mfg)$ is a graded deformation of the enveloping algebra $U(\mfs)$. 
\begin{proposition}\label{P:s->Yh}
 The assignment 
 \begin{equation}\label{s->Y}
 X_{ir}^\pm \mapsto x_{ir}^\pm, \; H_{ir}\mapsto h_{ir}\quad \forall \; i\in I \;\text{ and }\; r\geq 0
 \end{equation}
 extends to an isomorphism of graded $\C$-algebras $U(\mfs)\stackrel{\sim}{\rightarrow} Y_0(\mfg)$. 
\end{proposition}
\begin{proof}
Since the defining relations of $Y_0(\mfg)$ are of Lie type, it is isomorphic to $U(\mfs^\prime)$, where $\mfs^\prime$ is the Lie algebra generated by $\{x_{ir}^\pm, h_{ir}\}_{i\in I,r\geq 0}$ subject to the defining relations \eqref{s-HH}, \eqref{s-X+X-}, \eqref{s-xx}, in addition to the three relations 
 \begin{gather}
  [h_{i 0}, x^\pm_{j s} ] = \pm 2d_{ij} x^\pm_{j s}, \quad [h_{i, r+1}, x^\pm_{j s}]=[h_{i r}, x^\pm_{j, s+1}], \label{s-HX'}\\
  \sum_{\sigma\in S_m}
   [x^\pm_{i r_{\sigma(1)}}, [x^\pm_{i r_{\sigma(2)}}, \cdots,
       [x^\pm_{i, r_{\sigma(m)}}, x^\pm_{j s}] \cdots ]]
   = 0
   \quad \text{for $i\neq j$ and  $m=1-a_{ij}$}. \label{s-Serre'}
 \end{gather}
\eqref{s->Y} extends to an epimorphism of algebras $U(\mfs)\onto Y_0(\mfg)$ since \eqref{s-HX} follows from \eqref{s-HX'}.  To conclude that the assignment $x_{ir}^\pm\mapsto X_{ir}^\pm, \, h_{ir}\mapsto H_{ir}$
 extends to a homomorphism $Y_0(\mfg)\to U(\mfs)$ which is the inverse of the homomorphism $U(\mfs)\to Y_0(\mfg)$ defined by \eqref{s->Y}, it suffices to show that the relations of Definition \ref{D:s} imply \eqref{s-HX'} and \eqref{s-Serre'}. 
 
 Since \eqref{s-HX} implies \eqref{s-HX'}, we are left to deduce \eqref{s-Serre'} from Definition \ref{D:s}. We will prove the stronger result 
 \begin{equation}\label{s-Serre''}
  \left[X_{i,r_1}^\pm,\left[X_{i,r_2}^\pm,\ldots,\left[X_{i,r_{m}}^\pm,X_{js}^\pm\right]\cdots\right]\right]=0 
 \end{equation}
 for all $s,r_1,\ldots,r_m\geq 0$ and $i\neq j \in I$. 
 
 From \eqref{s-xx} and induction we obtain 
  \begin{equation}
  [X_{i,r+k}^\pm,X_{js}^\pm]=[X_{ir}^\pm,X_{j,s+k}^\pm]\quad \forall\; r,s,k\geq 0  \;\text{ and }\; i,j\in I. 
 \end{equation}
 This implies that for any fixed $n\geq 0$ and $k,s,r_1,\ldots,r_{n}\in \Z_{\ge 0}$, we have 
 \begin{gather*}
 \begin{aligned}
  \mathrm{ad}(X_{i,r_1+k}^\pm)\mathrm{ad}(X_{i,r_2}^\pm)\cdots &\mathrm{ad}(X_{i,r_n}^\pm)(X_{j,s}^\pm)\\&= \mathrm{ad}(X_{i,r_1}^\pm)\mathrm{ad}(X_{i,r_2}^\pm)\cdots \mathrm{ad}(X_{i,r_n}^\pm)(X_{j,s+k}^\pm).
 \end{aligned}
 \end{gather*}
After combining this with \eqref{s-serre}, relation \eqref{s-Serre''} becomes an immediate consequence.  \qedhere 
\end{proof}
\begin{remark}\label{R:flat}
When $\mfg$ is finite-dimensional, it is known that $Y_\hbar(\mfg)$ is a flat deformation of $U(\mfs)$. We will prove the analogous result for $\mfg$ of simply laced affine type in Theorem \ref{T:flat}.
\end{remark}
Recall the filtration $\{\mbF_k\}_{k\geq 0}$ on $Y(\mfg)$ defined at the end of Subsection \ref{ssec:defY}. Let $\bar{x}_{ir}^\pm$ and $\bar{h}_{ir}$ denote the images of $x_{ir}^\pm$ and $h_{ir}$ in $\mbF_r/\mbF_{r-1}\subset \gr Y(\mfg)$, where $\mbF_{-1}=\{ 0 \}$. 
The following result is immediate from the defining relations of $Y(\mfg)$. 
\begin{proposition}\label{P:weakPBW}
 The assignment 
 \begin{equation*}
  X_{ir}^\pm \mapsto \bar x_{ir}^\pm, \: H_{ir}\mapsto \bar h_{ir}\quad \forall\; i\in I\; \text{ and }\; r\geq 0
 \end{equation*}
 extends to an epimorphism of graded $\C$-algebras $\phi:U(\mfs)\onto \gr Y(\mfg)$. 
\end{proposition}
The statement that $\phi$ is injective is equivalent to the Poincar\'{e}-Birkhoff-Witt Theorem for the Yangian. 
For $\mfg$ of finite type this was proven in the early 1990's by Levendorskii \cite{Le} (see also Appendix B in  \cite{FT} and Proposition 2.2 in \cite{GRW}), but in the general setting this remains a conjecture. We will prove the injectivity of $\phi$ for $\mfg$ of simply laced affine type in Section \ref{sec:PBW}.

\subsection{The coproduct $\Delta_u$}\label{ssec:Delta_u}

The Yangian of a finite-dimensional simple Lie algebra is well-known to admit the structure of a Hopf algebra. In particular, it is equipped with a coassociative algebra homomorphism $\Delta:Y(\mfg)\to Y(\mfg)\otimes Y(\mfg)$, its coproduct. When the underlying simple Lie algebra is replaced with a more general Kac-Moody algebra, the formulas used to define $\Delta$ are no longer well-defined. However, it was shown in \cite{GNW} that, when $\mfg$ is affine, there is an algebra homomorphism $\Delta_u:Y(\mfg)\to (Y(\mfg)\otimes Y(\mfg))(\!(u)\!)$ which, in a strictly formal sense, has limit at $u=1$ which is in agreement with $\Delta$. The definition of $\Delta_u$ is contained in the following theorem. Set $\tilde h_{i1}=h_{i1}-\frac{1}{2}h_{i0}^2$ for all $i\in I$ and $\square(a)=a\otimes 1+1\otimes a$ for all $a\in Y(\mfg)$. 
\begin{theorem}[Thm.\ 6.2 of \cite{GNW}]\label{T:coprodu}
 Assume that the Cartan matrix $\mbA$ of $\mfg$ is of affine type, but not of type $A_1^{(1)}$ or $A_2^{(2)}$. Then there is an algebra homomorphism 
 \begin{equation*}
  \Delta_u:Y(\mfg)\to (Y(\mfg)\otimes Y(\mfg))(\!(u)\!)
 \end{equation*}
 uniquely determined by 
 \begin{gather} 
\begin{split}\label{Delta_u}
 \Delta_u(x_{i0}^\pm)=x_{i0}^\pm \otimes 1 + 1\otimes x_{i0}^\pm u^{\pm 1}, \quad \Delta_u(h_{i0})=\square(h_{i0}),\\
 \Delta_u(\tilde h_{i1})=\square (\tilde h_{i1})-  \sum_{\alpha\in \Delta_+^\mathrm{re}}(\alpha,\alpha_i)x_\alpha^-\otimes x_\alpha^+ u^{\mathrm{ht}(\alpha)},
 \end{split}
\end{gather}
for all $i\in I$, where $\Delta_+^{\mathrm{re}}$ is the set of positive real roots, $\mathrm{ht}(\sum_{i\in I}n_i \alpha_i)=\sum_{i\in I}n_i$, and $x_\alpha^\pm\in \mfg_{\pm \alpha}$ are such that $(x_\alpha^+,x_\alpha^-)=1$. 
\end{theorem}
The morphism $\Delta_u$ is not coassociative in the standard sense, but it satisfies the "twisted" coassociativity relation 
\begin{equation}\label{u-coas}
(\Delta_u\otimes \mathrm{id}) \circ \Delta_{uv}=(\mathrm{id}\otimes \Delta_v) \circ \Delta_{u}.
\end{equation}
By repeated application of $\Delta_u$, one can obtain an algebra homomorphism $\Delta_u^k:Y(\mfg)\to (Y(\mfg)^{\otimes (k+1)})(\!(u)\!)$ for each $k\geq 0$. However, due to the presence of the parameter $u$ and the twisted coassociativity property \eqref{u-coas}, this must be handled carefully. 

Given an associative unital $\C$-algebra $\mcA$ and a fixed $k\geq 1$, we denote by 
$\mcA(\!(u_k,\ldots,u_1)\!)$ the localization of $\mcA[\![u_k,\ldots,u_1]\!]$ at the multiplicative set 
\begin{equation*}
 S=\{u_k^{m_k}\cdots u_1^{m_1}\,:\, m_a\geq 0\}.
\end{equation*}
Equivalently, $\mcA(\!(u_k,\ldots,u_1)\!)$ can be realized as the subspace of $\mcA[\![u_k^{\pm 1},\ldots,u_1^{\pm 1}]\!]$ consisting of elements
\begin{equation*}
 \sum_{\ell_1,\ldots,\ell_k \in \Z}a_{\ell_k,\ldots,\ell_1} u_k^{\ell_k}\cdots u_1^{\ell_1}
\end{equation*}
for which there exists $N\geq 0$ such that, for any $1\le m\le k$, $a_{\ell_k,\ldots,\ell_1}=0$ whenever $\ell_m<-N$, with product obtained by extending the usual multiplication of formal series in $\mcA[\![u_k,\ldots,u_1]\!]$. The key feature of this algebra we will exploit is that
\begin{equation}\label{ev}
\mathrm{ev}_{u,k}: f(u_k,\ldots,u_1)\mapsto f(u,\ldots,u)\quad \forall \; f(u_k,\ldots,u_1)\in \mcA(\!(u_k,\ldots,u_1)\!). 
\end{equation}
determines an algebra homomorphism $\mathrm{ev}_{u,k}:\mcA(\!(u_k,\ldots,u_1)\!)\to \mcA(\!(u)\!)$. 

\begin{NB}
 The evaluation $\mcA[u_k,\ldots,u_1]\to \mcA[u]$, $u_m\mapsto u$ for all $m$, extends by continuity to $\mcA[\![u_k,\ldots,u_1]\!]\to\mcA[\![u]\!]\subset \mcA(\!(u)\!)$. As it sends elements of $S$ to units, the universal property of localization gives $\mathrm{ev}_{k,u}$. 
\end{NB}
Let us define $\mcA(\!(u_{k})\!)(\!(u_{k-1})\!)\cdots (\!(u_1)\!)$ inductively as $\Big(\cdots\big(\mcA(\!(u_{k})\!)\big)(\!(u_{k-1})\!)\cdots \Big) (\!(u_1)\!)$.
To define $\Delta_u^k$ we will make use of auxiliary morphisms 
\begin{equation*}
\Delta_{u_1,\ldots u_k}:Y(\mfg)\to (Y(\mfg)^{\ot (k+1)})(\!(u_{k})\!)(\!(u_{k-1})\!)\cdots (\!(u_1)\!)
\end{equation*}
which are defined recursively as follows:  $\mathrm{id}^{\ot (k-1)}\otimes \Delta_{u_k}$ extends to a morphism 
\begin{equation*}
(Y(\mfg)^{\ot k})(\!(u_{k-1})\!)\cdots (\!(u_1)\!)\to (Y(\mfg)^{\ot (k+1)})(\!(u_{k})\!)(\!(u_{k-1})\!)\cdots (\!(u_1)\!),
\end{equation*}
and the composition of this morphism with $\Delta_{u_1,\ldots,u_{k-1}}$ is precisely $\Delta_{u_1,\ldots,u_k}$. Inductively, we find that 
\begin{gather}
\begin{split} \label{Delta_u^k:1}
 \Delta_{u_1,\ldots,u_k}(h_{i0})=\sum_{a=1}^{k+1}(h_{i0})_a,\quad \Delta_{u_1,\ldots,u_k}(x_{i0}^{\pm})=\sum_{a=1}^{k+1}(x_{i0}^\pm)_a u_1^{\pm 1}\cdots u_{a-1}^{\pm 1},\\
 \Delta_{u_1,\ldots,u_k}(\tilde h_{i1})=\sum_{a=1}^{k+1}(\tilde h_{i1})_a-\sum_{a<b}\sum_{\alpha\in \Delta_+^{\mathrm{re}}}(\alpha,\alpha_i)(x_\al^-)_a(x_\al^+)_b u_a^{\mathrm{ht}(\alpha)}\cdots u_{b-1}^{\mathrm{ht}(\alpha)},
\end{split}
\end{gather}
where $(X)_a=1^{\ot (a-1)}\ot X \ot 1^{\ot (k+1-a)}$, and the product $u_1^{\pm 1}\cdots u_{a-1}^{\pm 1}$ with $a=1$ is understood to equal $1$. Consequently, $\mathrm{Image}(\Delta_{u_1,\ldots,u_k})\subset Y(\mfg)^{\ot (k+1)}(\!(u_k,\ldots u_1)\!)$, and we may therefore set
\begin{equation}\label{Delta_u^k:2}
 \Delta_u^k=\mathrm{ev}_{u,k}\circ \Delta_{u_1,\ldots,u_k}: Y(\mfg)\to  Y(\mfg)^{\ot (k+1)}(\!(u)\!) \quad \forall \; k\geq 1,
\end{equation}
 where $\mathrm{ev}_{u,k}$ is as in \eqref{ev} with $\mcA=Y(\mfg)^{\ot (k+1)}$. 
 
 The explicit formulas \eqref{Delta_u^k:1} imply that $\Delta_u^k$ is filtered in the sense that
\begin{equation}\label{Delta_u^k:3}
 \Delta_{u}^k(\mbF_\ell)\subset (\mbF_\ell(Y(\mfg)^{\ot (k+1)}))(\!(u)\!),
\end{equation}
where $\mbF_\ell(Y(\mfg)^{\ot (k+1)})=\sum_{a_1+\ldots+a_{k+1}=\ell} \mbF_{a_1}\ot \cdots \ot \mbF_{a_{k+1}}$. By \eqref{Delta_u^k:1}, the associated graded morphism $\gr\Delta_u^k$ has image contained in $\gr Y(\mfg)^{\otimes (k+1)}[u^{\pm 1}]$ for each $k$: 
\begin{equation}\label{grDelta_u^k}
 \gr\Delta_u^k: \gr Y(\mfg)\to \gr(Y(\mfg)^{\otimes (k+1)})[u^{\pm 1}] = \gr(Y(\mfg))^{\otimes (k+1)}[u^{\pm 1}].
\end{equation}
The family of filtered morphisms $\{\Delta_u^k\}_{k\geq 1}$ will play a decisive role in the proof of the Poincar\'{e}-Birkhoff-Witt Theorem in Section \ref{sec:PBW}, as will the analogous morphisms $\{\Delta_{\mfs,u}^k\}_{k\geq 1}$ for the enveloping algebra $U(\mfs)$, which we define now. 

Let $\Delta_\mfs$ denote the standard coproduct on $U(\mfs)$. The assignment $X_{ir}^\pm\mapsto  u^{\pm 1}X_{ir}^\pm$, $H_{ir}\mapsto H_{ir}$ for all $i\in I$ and $r\geq 0$ extends to an algebra homomorphism $\mathrm{s}_{u}:U(\mfs)\to U(\mfs)[u^{\pm 1}]$, and we may set 
\begin{equation*}
 \Delta_{\mfs,u}=(\mathrm{id}\otimes \mathrm{s}_u)\circ \Delta_{\mfs}:U(\mfs)\to (U(\mfs)\otimes U(\mfs))[u^{\pm 1}].
\end{equation*}
The morphisms $\Delta_{\mfs,u}^k:U(\mfs)\to U(\mfs)^{\otimes (k+1)}[u^{\pm 1}]$ are now constructed in exactly the same way as $\Delta_u^k$ (see \eqref{Delta_u^k:2}). On generators, we have 
\begin{equation*}
 \Delta_{\mfs,u}^k(H_{ir})=\sum_{a=1}^{k+1}(H_{ir})_a,\quad \Delta_{\mfs,u}^k(X_{ir}^\pm)=\sum_{a=1}^{k+1}(X_{ir}^\pm)_a u^{\pm(a-1)}.
\end{equation*}
In particular, we have the following commutative diagram:

\begin{equation} 
\begin{alignedat}{50}\label{Delta_u:com}
\xymatrix{
U(\mfs) \ar[rrr]^{ \Delta_{\mfs,u}^k } \ar[d]_{\phi} & & &  U(\mfs)^{\otimes {(k+1)}}[u^{\pm 1}] \ar[d]^{\phi^{\otimes (k+1)}}
\\
\gr Y(\mfg) \ar[rrr]_{ \gr\Delta_u^k } & & & \gr Y(\mfg)^{\otimes {(k+1)}}[u^{\pm 1}]
}
\end{alignedat}
\end{equation}
The map $\phi$ is the one given in Proposition \ref{P:weakPBW}.

\section{The centrally extended Yangian double of \texorpdfstring{$\mfg$}{g}}\label{Sec:DYg}

In this section we introduce the centrally extended Yangian double associated to $\mfg$ and study its basic algebraic properties. 

\subsection{Definition of the Yangian double}\label{ssec:defDY}
Let $\delta(w,z)=\sum_{r\in \Z} w^r z^{-r-1}\in \C[\![w^{\pm 1}, z^{\pm 1}]\!]$ denote the formal delta function.  Equivalently, 
\begin{equation}\label{eq:deltaexp} 
 \delta(w,z) = \frac{z^{-1}}{1-z^{-1}w} + \frac{w^{-1}}{1-w^{-1}z},\quad \text{ where }\quad \frac{x^{-1}}{1-x^{-1}y}=\sum_{k\geq 0}y^k x^{-k-1}.
\end{equation}
\begin{definition}\label{DY:def}
The centrally extended Yangian double $DY_{\hbar}^\msc(\mfg)$ is the $\C[\hbar]$-algebra generated by the coefficients $\{h_{ir},x_{ir}^\pm\}_{i\in I,r\in \Z}$ of 
\[x_i^{\pm}(z) = \sum_{r\in \Z} x_{ir}^{\pm} z^{-r-1},\; h_i^+(z) = 1+\hbar\!\!\sum_{r\in \Z_{\geq 0}} h_{ir}z^{-r-1},\; h_i^-(z) = 1-\hbar\!\!\sum_{r\in \Z_{<0}} h_{ir} z^{-r-1},\]
for all $i\in I$, together with an element $\msc$, which are subject to the defining relations 
\begin{gather}
\tfrac{1}{\hbar}[\msc,h_i^\pm(z)]=0=[x_i^\pm(z),\msc] , \label{DY:c}\\
\tfrac{1}{\hbar^2}[h_{i}^{\pm}(z),\,h_j^{\pm}(w)] = 0 , \label{DY:hh}\\
\tfrac{1}{\hbar^2}\left(\left((z-w)^2-(\msc_{ij}^-)^2\right)h_i^+(z)\,h_j^-(w) -\left((z-w)^2-(\msc_{ij}^+)^2\right)h_j^-(w)\,h_i^+(z)\right)=0 , \label{DY:hhh}\\
\tfrac{1}{\hbar}\left(\left(z-w\mp\msc_{ij}^-\right)h_i^+(z)\,x_j^{\pm}(w)-\left(z-w\pm\msc_{ij}^+\right)x_j^{\pm}(w)\,h_i^+(z)\right)=0, \label{DY:xh+}\\
\tfrac{1}{\hbar}\left(\left(z-w\mp\hbar d_{ij}\right)h_i^-(z)\,x_j^{\pm}(w)-\left(z-w\pm\hbar d_{ij}\right) 
x_j^{\pm}(w)\,h_i^-(z)\right)=0 , \label{DY:xh-}\\
\left(z-w\mp\hbar d_{ij}\right) x_{i}^{\pm}(z)\,x_{j}^{\pm}(w) = \left(z-w\pm\hbar d_{ij}\right) x_{j}^{\pm}(w)\,x_{i}^{\pm}(z) ,  \label{DY:xx}\\
[x_i^+(z),x_j^-(w)] = \tfrac{\delta_{ij}}{\hbar} \left(\delta({w + \hbar \msc},{z}) h_i^+\!\left( w+\tfrac{\hbar \msc}{2} \right) - \delta({w},{z}) h_i^-\!\left( z \right) \right) , \label{DY:xxh}\\
\sum_{\sigma \in S_{m}} [x_i^{\pm}(z_{\sigma(1)}), [x_i^{\pm}(z_{\sigma(2)}), \cdots, [x_i^{\pm}(z_{\sigma(m)}),x_j^{\pm}(w)] \cdots]] = 0, \label{DY:serre}
\end{gather}
where $\msc_{ij}^\pm=\hbar d_{ij}\pm \frac{\hbar \msc}{2}$ and in the last relation $i\neq j$ and $m=1-a_{ij}$.
\end{definition}
For each $\kappa\in \C$, we define the \textit{Yangian double at level $\kappa$} to be the $\C[\hbar]$-algebra
\begin{equation*}
 DY_\hbar^{\kappa}(\mfg)=DY_{\hbar}^\msc(\mfg)/(\msc-\kappa)DY_{\hbar}^\msc(\mfg).
\end{equation*}
\begin{remark}
Even though the relations  \eqref{DY:c}-\eqref{DY:xh-} and \eqref{DY:xxh} involve negative powers of $\hbar$, this is not the case for the corresponding relations among the generators. (See Lemma \ref{DY':def}.) Not dividing by $\hbar$ could create $\hbar$-torsion elements.
\end{remark}
\begin{remark}
The practice of calling $DY_\hbar^\msc(\mfg)$ the "Centrally extended Yangian double" is explained by the following: when $\mfg$ is finite-dimensional, $DY_\hbar^0(\mfg)$ has been conjectured to be equal, after completion, to the Hopf algebra double of $Y_\hbar(\mfg)$ \cite{KhTo}, whereas $DY_\hbar^\msc(\mfg)$ has been conjecturally described, also after completion, as a quotient of the Hopf algebra double of $Y_\hbar(\mfg)\otimes \C[\msc]$ by a derivation \cite{Kh}. These conjectures have been proven for $\mfg=\mfsl_2$: see \cite[Prop.\ 2.1 (ii)]{KhTo} and \cite[Thm.\ 3.1]{Kh}. 

Although this interpretation of $DY_\hbar^\msc(\mfg)$ does not extend beyond the finite case, Definition \ref{DY:def} is a natural extension of the definitions found in the literature (see in particular \cite[\S 6]{DiKh} and \cite[Cor.\ 3.4]{Io}). 
\end{remark}
The following lemma is straightforward. 
\begin{lemma}\label{DY':def}
For each $i\in I$, set $\wt h_i^\pm(z) =\pm \frac{1}{\hbar}(h_i^\pm(z)-1)$. Then the relations of Definition \ref{DY:def} are equivalent to
\begin{gather}
[\msc,\wt h_i^\pm(z)]=0=[x_i^\pm(z),\msc] , \label{DY':c}\\
\wt h_{i}^{\pm}(z)\,\wt h_j^{\pm}(w) = \wt h_j^{\pm}(w)\,\wt h_{i}^{\pm}(z) , \label{DY':hh}\\
\begin{aligned}\label{DY':hhh}
\left((z-w)^2-(\msc_{ij}^+)^2\right)&\wt h_j^-(w)\,\wt h_i^+(z)-\left((z-w)^2-(\msc_{ij}^-)^2\right)\wt h_i^+(z)\,\wt h_j^-(w) \\
                                                       &=-2d_{ij}\msc+2\hbar d_{ij}(\wt h_j^-(w)-\wt h_i^+(z))\msc, 
\end{aligned}\\
\left(z-w\mp \msc_{ij}^-\right)\wt h_i^+(z)\,x_j^{\pm}(w)-\left(z-w\pm \msc_{ij}^+\right)x_j^{\pm}(w)\,\wt h_i^+(z)=\pm 2d_{ij} x_j^\pm(w), \label{DY':xh+}\\
\left(z-w\pm\hbar d_{ij}\right) 
x_j^{\pm}(w)\,\wt h_i^-(z)-\left(z-w\mp\hbar d_{ij} \right)\wt h_i^-(z)\,x_j^{\pm}(w)=\pm 2 d_{ij}x_j^\pm(w) , \label{DY':xh-}\\
\left(z-w\mp\hbar d_{ij}\right) x_{i}^{\pm}(z)\,x_{j}^{\pm}(w) = \left(z-w\pm\hbar d_{ij}\right) x_{j}^{\pm}(w)\,x_{i}^{\pm}(z) ,  \label{DY':xx}\\
\begin{aligned} \label{DY':xxh}
[x_i^+(z),x_j^-(w)] =& \delta_{ij}\tfrac{1}{\hbar}\left(\delta({w + \hbar \msc},{z})-\delta({w},{z})\right)\\
                     &+\delta_{ij}\left(\delta({w + \hbar \msc},{z})\wt h_i^+(w+\tfrac{\hbar \msc}{2})+\delta({w},{z})\wt h_i^-(z)\right),
\end{aligned}\\
\sum_{\sigma \in S_{m}} [x_i^{\pm}(z_{\sigma(1)}), [x_i^{\pm}(z_{\sigma(2)}), \cdots, [x_i^{\pm}(z_{\sigma(m)}),x_j^{\pm}(w)] \cdots]] = 0, \label{DY':serre}
\end{gather}
where in the last relation $i\neq j$ and $m=1-a_{ij}$.
\end{lemma}

It is not difficult to deduce from these relations that $DY_\hbar^\msc(\mfg)$ is a $\Z$-graded algebra with grading determined by 
\begin{equation*}
\deg \hbar=1,\quad \deg \msc=0 \quad \text{ and }\quad \deg x_{ir}^\pm=\deg h_{ir}=r \quad \forall \; i\in I,r\in \Z.
\end{equation*}
\begin{NB}
 Let $\mbG_k$ denote the subspace spanned by monomials of degree $k$, and for any fixed $a\in \Z$ consider the subspace $\mbD_a=\prod_{r,s\in\Z}\mbG_{r+s+a}z^{-r-1}w^{-s-1}$ 
 of $DY_\hbar^\msc(\mfg)[\![z^{\pm 1},w^{\pm 1}]\!]$. Then $(z-w)\mbD_a\subset \mbD_{a+1},\hbar \mbD_a\subset \mbD_{a+1}$, and the statement that 
 the defining relations of Lemma \ref{DY':def} are homogeneous can be equivalently stated as follows: For each fixed relation $R=R'$, both $R$ and $R'$ belong to $\mbD_a$ for a fixed $a$. This is not hard to verify.
\end{NB}
Next, for each $\zeta\in \C$ we introduce a $\C$-algebra
\begin{equation*}
 DY_\zeta^\msc(\mfg)=DY_\hbar^\msc(\mfg)/(\hbar-\zeta)DY_\hbar^\msc(\mfg),
\end{equation*}
and we abbreviate $DY^\msc(\mfg)=DY_1^\msc(\mfg)$. 
Note that, analogously to $Y_\zeta(\mfg)$, $DY_\zeta^\msc(\mfg)$ for $\zeta\in \C^\times$ is precisely the $\C$-algebra generated by $\{x_{ir}^\pm,h_{ir}\}_{i\in I,r\in \Z}$ and $\msc$ subject to the defining relations of Definition \ref{DY:def} with $\hbar$ replaced by $\zeta$. For each $\zeta \in \C^\times$, the assignment 
\begin{equation*}
x_{ir}^\pm \mapsto \zeta^{-r} x_{ir}^\pm, \; h_{ir} \mapsto  \zeta^{-r}h_{ir}, \; \msc \mapsto \msc,
\end{equation*}
or equivalently $x_i^\pm(z)\mapsto \zeta x_i^\pm(\zeta z), \, \wt h_i^\pm(z) \mapsto  \zeta \wt h_i^\pm (\zeta z), \, \msc \mapsto \msc $, extends to an isomorphism of algebras between $DY^\msc(\mfg)$ and $DY_\zeta^\msc(\mfg)$. With this in mind, we will henceforth focus primarily on the $\C[\hbar]$-algebra $DY^{\msc}_\hbar(\mfg)$ and the $\C$-algebra $DY^\msc(\mfg)$.

The degree assignments  $\deg \msc=0$ and $\deg x_{ir}^\pm=\deg h_{ir}=r$ determine a $\Z$-filtration (but not a gradation) on $DY^\msc(\mfg)$. For each $k\in \Z$, let $\mbF_k^D$ denote the subspace of $DY^\msc(\mfg)$ spanned by monomials of degree $\leq k$, and let 
\[\mathrm{gr}_\Z DY(\mfg)=\bigoplus_{k\in \Z}\mbF_k^D/\mbF_{k-1}^D\]
denote the corresponding associated graded algebra. 

\subsection{Classical limit}\label{ssec:clDY}
\begin{definition}\label{D:t}
Define $\mft$ to be the Lie algebra generated by an element $\mathbf{C}$ together with the coefficients $\{X_{ir}^\pm,H_{ir}\}_{i\in I,r\in \Z}$ of  
\[X_i^\pm(z)=\sum_{k\in \Z}X_{ik}^\pm z^{-k-1}\quad \text{ and }\quad H_i(z)=\sum_{k\in \Z}H_{ik}z^{-k-1}\quad \forall \; i\in I, \]
which are
subject to the defining relations
 \begin{gather}
  [\mathbf{C},H_i(z)]=0=[X_i^\pm(z),\mathbf{C}], \label{t:C}\\
  [H_i(z),H_j(w)]=-2d_{ij}\delta_z({z},{w})\mathbf{C}, \label{t:HH}\\ 
  [H_i(z),X_j^\pm(w)]=\pm 2d_{ij}\delta({z},{w})X_j^\pm(w), \label{t:HX}\\
  [X_i^+(z),X_j^-(w)]=\delta_{ij}\left(\delta({z},{w})H_i(z)-\delta_z({z},{w})\mathbf{C}\right), \label{t:X+X-}\\
  (z-w)[X_i^\pm(z),X_j^\pm(w)]=0,\label{t:XX}\\
  \mathrm{ad}(X_{i0}^\pm)^{1-a_{ij}}(X_j^\pm(z))=0, \label{t:Serre}
 \end{gather}
where $\delta_z({z},{w})=\frac{d}{dz}\delta({z},{w})=\sum_{r\in \Z}rz^{r-1}w^{-r-1}$.
\end{definition}
Note that the degree assignment 
\begin{equation*}
\deg \mathbf{C}=0 \quad \text{ and }\quad \deg X_{ir}^\pm=\deg H_{ir}=r \quad \forall \; i\in I,r\in \Z
\end{equation*}
makes $\mft$ into a $\Z$-graded Lie algebra. Additionally, just as $\mfs$ is an extension of the current algebra $\mfg^\prime[t]$, the Lie algebra $\mft$ is an extension of the loop algebra $\mfg^\prime[t^{\pm 1}]$. In the notation of \eqref{s->g[t]}, the assignment 
\begin{equation*}
\mathbf{C} \mapsto 0, \;  X_{ir}^\pm \mapsto x_i^\pm\ot t^r, \; H_{ir}\mapsto h_i\ot t^r \quad \forall \; i\in I \;\text{ and }\; r\in \Z
\end{equation*}
defines a surjective Lie algebra morphism $\mft\onto \mfg^\prime[t^{\pm 1}]$. We will return to this observation in Section \ref{Sec:aff}.

We now give the analogue of Proposition \ref{P:s->Yh} for the Yangian double. Let $\wt h_i^\pm(z)$ be as in Lemma \ref{DY':def} (now viewed in $DY_0^\msc(\mfg)[\![z^{\pm 1}]\!]$) and set $\wt{h}_i(z) = \wt h_i^+(z)+\wt h_i^-(z)= \sum_{r\in\Z} h_{ir}z^{-r-1}$.
\begin{proposition}\label{P:mftiso}
 The assignment 
 \begin{equation}
   X_i^\pm(z) \mapsto  x_i^\pm(z), \; H_i(z) \mapsto \wt{h}_i(z), \; \mathbf{C}\mapsto \msc \quad \forall \; i\in I
 \end{equation}
 extends to an isomorphism of graded $\C$-algebras $U(\mft)\stackrel{\sim}{\rightarrow} DY_0^\msc(\mfg)$. 
\end{proposition}
\begin{proof}
 By definition, $DY_0^\msc(\mfg)$ is the $\C$-algebra generated by the coefficients of $x_i^\pm(z)$, $\wt{h}_i(z)$ and the central element $ \msc$, which are subject to the relations of Definition \ref{DY:def} with $\hbar$ replaced by $0$. Lemma \ref{DY':def} implies that, in addition to the centrality of $ \msc$, these relations are
 \begin{gather}
\wt h_{i}^{\pm}(z)\,\wt h_j^{\pm}(w) = \wt h_j^{\pm}(w)\, \wt h_{i}^{\pm}(z) , \label{D0:hh}\\
(z-w)^2[\wt h_j^-(w),\wt h_i^+(z)]=-2d_{ij}\msc \label{D0:hh2} \\
(z-w)[\wt h_i^+(z), x_j^{\pm}(w)]=\pm 2d_{ij}  x_j^\pm(w), \label{D0:xh+}\\
(z-w)[ x_j^{\pm}(w),\wt h_i^-(z)]=\pm 2 d_{ij} x_j^\pm(w) , \label{D0:xh-}\\
\left(z-w\right) [ x_{i}^{\pm}(z), x_{j}^{\pm}(w)] = 0,  \label{D0:xx}\\
\begin{aligned} \label{D0:xxh}
[ x_i^+(z), x_j^-(w)] =& \delta_{ij}\lim_{\hbar\to 0}\tfrac{1}{\hbar}\left(\delta({w + \hbar \msc},{z})-\delta({w},{z})\right) +\delta_{ij}\delta({w},{z}) \wt{h}_i(w),
\end{aligned}\\
\sum_{\sigma \in S_{m}} [ x_i^{\pm}(z_{\sigma(1)}), [ x_i^{\pm}(z_{\sigma(2)}), \cdots, [ x_i^{\pm}(z_{\sigma(m)}), x_j^{\pm}(w)] \cdots]] = 0. \label{D0:serre}
\end{gather}
It therefore suffices to show that these relations are equivalent to the defining relations of Definition \ref{D:t} (with $(H_{ir},X_{ir}^\pm)$ replaced by  
$(h_{ir}, x_{ir}^\pm)$ and $\mathbf{C}$ by $ \msc$). 

\noindent\textit{Step 1:} (\eqref{D0:hh},\eqref{D0:hh2}) $\iff$ \eqref{t:HH}. 

Multiplying \eqref{D0:hh2} by $\sum_{k\geq 0}(k+1)w^k z^{-k-2}=\left(\frac{z^{-1}}{1-z^{-1}w}\right)^2$ yields the relation 
\begin{equation*}
 [\wt h_j^-(w), \wt h_i^+(z)]=-2d_{ij}\left(\frac{z^{-1}}{1-z^{-1}w}\right)^2   \msc.
\end{equation*}
Combining this with \eqref{D0:hh} and using $\wt h_i(z)=\wt h_i^+(z)+\wt h_i^-(z)$, we obtain 
\begin{equation*}
 [\wt h_j(w),\wt h_i(z)]=2d_{ij}\left(\left(\frac{w^{-1}}{1-w^{-1}z}\right)^2 -\left(\frac{z^{-1}}{1-z^{-1}w}\right)^2 \right)\msc=2d_{ij}\delta_z({z},{w}) \msc.
\end{equation*}
Switching $i\leftrightarrow j$ and $z\leftrightarrow w$ yields \eqref{t:HH}. 

Conversely, taking the $z^{-r-1}w^{-s-1}$ coefficient of \eqref{t:HH} gives 
\begin{equation}\label{t:HH'}
[h_{ir},h_{js}] =2rd_{ij}\delta_{r,-s} \msc.
\end{equation}
Multiplying both sides by $z^{-r-1}w^{-s-1}$ and taking the sum separately over $r,s\geq 0$ and $r,s<0$ gives \eqref{D0:hh}. 

Switching $i$ and $j$ in \eqref{t:HH'}, multiplying both sides by $w^{-r-1}z^{-s-1}$ and taking the sum over $(r,s)\in \Z_{<0}\times \Z_{\geq 0}$ yields 
\begin{equation*}
 [\wt{h}_j^-(w),\wt{h}_i^+(z)]=-2d_{ij}\left(\frac{z^{-1}}{1-z^{-1}w}\right)^2 \msc.
\end{equation*}
Multiplying both sides by $(z-w)^2$ gives \eqref{D0:hh2}.

\noindent \textit{Step 2: } (\eqref{D0:xh+},\eqref{D0:xh-}) $\iff$ \eqref{t:HX}.

Multiplying  \eqref{D0:xh+} by $\sum_{k\geq 0} w^k z^{-k-1}$ and \eqref{D0:xh-} by $\sum_{k\geq 0}z^k w^{-k-1}$ gives 
\begin{equation*}
 [\wt h_i^+(z), x_j^\pm(w)]=\pm \frac{2d_{ij} z^{-1}}{1-z^{-1}w} x_j^\pm(w) \quad \text{ and }\quad [\wt h_i^-(z), x_j^\pm(w)]=\pm\frac{2d_{ij} w^{-1}}{1-w^{-1}z} x_j^\pm(w).
\end{equation*}
Adding these two relations together gives \eqref{t:HX}. Conversely, taking the $z^{-r-1}w^{-s-1}$ coefficient of \eqref{t:HX} yields 
\begin{equation*}
[h_{ir}, x_{js}^\pm]=\pm 2d_{ij}  x_{j,r+s}^\pm.
\end{equation*}
Multiplying both sides of this equality by $(z-w)z^{-r-1}w^{-s-1}$ and taking the sum $r\geq 0$ and $s\in \Z$ gives 
\begin{equation*}
(z-w)[\wt h_i^+(z), x_j^\pm(w)]=\pm 2d_{ij}\! \sum_{r\geq 0, s\in \Z}  x_{j,r+s}^\pm (z^{-r}w^{-s-1}-z^{-r-1}w^{-s})=\pm 2d_{ij} x_j^\pm (w),
\end{equation*}
which is precisely \eqref{D0:xh+}. The proof that \eqref{t:HX} implies \eqref{D0:xh-} is similar. 

\noindent \textit{Step 3: } \eqref{D0:xxh} $\iff$ \eqref{t:X+X-}, and \eqref{D0:xx}$\iff$ \eqref{t:XX}.

The equivalence of \eqref{D0:xx} with \eqref{t:XX} is immediate. To prove the \eqref{D0:xxh} $\iff$ \eqref{t:X+X-}, it suffices to show that 
\begin{equation*}
\lim_{\hbar\to 0}\tfrac{1}{\hbar}\left(\delta({w + \hbar \msc},{z})-\delta({w},{z})\right) =\delta_w({w},{z})\msc=-\delta_z({z},{w})\msc,
\end{equation*}
which can be verified directly.

\begin{NB}
Fixing $r>0$, we have 
\begin{gather*}
 (w+\hbar \msc)^r-w^r=\hbar \msc\sum_{k=0}^{r-1}\binom{r}{k}(\hbar \msc)^{r-k-1} w^k, \\
 (w+\hbar \msc)^{-r}-w^{-r}=-\hbar \msc\sum_{k\geq 1} \binom{r+k-1}{k}(\hbar \msc)^{k-1}w^{-k-r}.
\end{gather*}
The identity above follows readily from these two equalities together with $\delta_w({w},{z})=-\delta_z({z},{w})$. 
\end{NB}

\noindent \textit{Step 4:} \eqref{D0:serre} $\implies$ \eqref{t:Serre}, and (\eqref{t:HX},\eqref{t:XX},\eqref{t:Serre}) $\implies$ \eqref{D0:serre}. 

The first implication is obvious. The second implication is proven in the same way as its $\mfs$-analogue in Proposition \ref{P:s->Yh}. 
\end{proof}
Recall the filtration $\{\mbF_k^D\}_{k\in \Z}$ on $DY^\msc(\mfg)$ defined at the end of Subsection \ref{ssec:defDY}. 
Let $\bar x_{ir}^\pm, \bar h_{ir}$ denote the images of $x_{ir}^\pm,h_{ir}$ in $\mbF_r^D/\mbF_{r-1}^D$ and $\bar \msc$ denote the image of $\msc$ in $\mbF_0^D/\mbF_{-1}^D$.

Similar verifications to those carried out in the proof of the previous proposition allow us to deduce the following analogue of Proposition \ref{P:weakPBW}.
\begin{proposition}\label{P:wPBW-DY}
 The assignment 
 \begin{equation*}
  X_{ir}^\pm \mapsto \bar x_{ir}^\pm, \; H_{ir} \mapsto \bar h_{ir}, \; \mathbf{C}\mapsto  \bar \msc\quad \forall\; i\in I\; \text{ and }\; r\in \Z
 \end{equation*}
 extends to an epimorphism of graded $\mathbb{C}$-algebras $\phi_D:U(\mft)\onto \mathrm{gr}_\Z DY^\msc(\mfg)$. 
\end{proposition}
Like the epimorphism $\phi:U(\mfs)\onto \gr Y(\mfg)$ of Proposition \ref{P:weakPBW}, we expect $\phi_D$ to be an isomorphism for general $\mfg$. However, the injectivity of $\phi_D$ will
not be considered in this paper.  

\subsection{From the Yangian to its double}\label{ssec:Y->DY}
We conclude this section by offering a more precise relation between $Y_\hbar(\mfg)$ and $DY_\hbar^\msc(\mfg)$. Let $\boldsymbol{x}_i^\pm(z)=\sum_{r\geq 0}x_{ir}^\pm z^{-r-1}\in DY_\hbar^\msc(\mfg)[\![z^{-1}]\!]$ for each $i\in I$. 
\begin{proposition}\label{P:iota}
 The assignment 
 \begin{equation}
  \iota_\hbar: x_i^\pm(z) \mapsto \boldsymbol{x}_i^\pm\!\left(z\pm \tfrac{\hbar \msc}{2}\right), \; h_i(z)\mapsto \wt h_i^+(z) \quad \forall\; i\in I \label{iota}
 \end{equation}
 extends to a morphism of $\C[\hbar]$-algebras $\iota_\hbar: Y_\hbar(\mfg)\mapsto DY_\hbar^\msc(\mfg)$. The composition of $\iota_\hbar$ with the projection 
 $DY_\hbar^\msc(\mfg)\onto DY^\msc(\mfg)$ induces a morphism of $\C$-algebras $\iota: Y(\mfg)\to DY^\msc(\mfg)$. 
\end{proposition}
\begin{proof}
 To distinguish between the generating series of $Y_\hbar(\mfg)$ and $DY_\hbar^\msc(\mfg)$, we will temporarily denote the series $x_i^\pm(z)\in Y_\hbar(\mfg)[\![z^{-1}]\!]$ from Proposition \ref{P:Y-op} by $\mcX_i^\pm(z)$. We will prove that $\iota_\hbar$ preserves the defining relations of $Y_\hbar(\mfg)$ provided by Proposition \ref{P:Y-op}.

 It is immediate that the relations \eqref{DY:hh} and \eqref{DY:serre} imply that 
$\wt h_i^+(z)$ and $\wt{\boldsymbol{x}}_i^\pm(z)=\boldsymbol{x}_i^\pm\!\left(z\pm \tfrac{\hbar \msc}{2}\right)$ satisfy the defining relations \eqref{Y:hh} and \eqref{Y:serre}, respectively, of $Y_\hbar(\mfg)$ (with $h_i(z)$ replaced by $\wt h_i^+(z)$ and $\mcX_i^\pm(z)$ replaced by $\wt{\boldsymbol{x}}_i^\pm(z)$).

Multiplying \eqref{DY:xh+} by $z^{-1}$ and taking the residue at $z=0$ gives $[h_{i0},x_j^\pm(w)]=\pm 2 d_{ij} x_j^\pm(w)$, and thus 
\begin{equation}\label{[h_i,x(w)]}
 [h_{i0},\wt{\boldsymbol{x}}_j^\pm(w)]=\pm 2 d_{ij} \wt{\boldsymbol{x}}_j^\pm(w). 
\end{equation}
Taking instead the $z^{-r-1}w^{-s-1}$ coefficient of \eqref{DY:xh+}, we obtain 
\begin{equation*}
 [h_{i,r+1},x_{js}^\pm]-[h_{ir},x_{j,s+1}^\pm]=\pm(\msc_{ij}^- h_{ir} x_{js}^\pm + \msc_{ij}^+ x_{js}^\pm h_{ir}).
\end{equation*}
Multiplying both sides by $z^{-r-1}w^{-s-1}$, taking the sum over $r,s\geq 0$, we obtain 
\begin{equation*}
 (z-w\mp \msc_{ij}^-)\wt h_i^+(z)\boldsymbol{x}_j^\pm(w)-(z-w\pm \msc_{ij}^+)\boldsymbol{x}_j^{\pm}(w)\wt h_i^+(z)=[h_{i0},\boldsymbol{x}_j^\pm(w)]-[\wt h_i^+(z),x_{j0}^\pm].
\end{equation*}
Substituting in the relation \eqref{[h_i,x(w)]} and applying $w\mapsto w\pm \frac{\hbar \msc}{2}$ yields \eqref{Y:xh}.

The proof that \eqref{DY:xx} implies \eqref{Y:xx} with $\mcX_i^\pm(z)$ and $\mcX_j^\pm(w)$ replaced by $\wt{\boldsymbol{x}}_i^\pm(z)$ and $\wt{\boldsymbol{x}}_j^\pm(w)$, respectively, is similar and  will be omitted. 

It thus remains to see that the assignment \eqref{iota} preserves the relation \eqref{Y:xxh}. By Proposition \ref{P:Yxxh}, it suffices to prove 
\begin{equation}\label{xxh:suf}
[x_{i0}^+,\wt{\boldsymbol{x}}_j^-(w)]=\delta_{ij} \wt h_i^+(w) \quad \forall \; i,j\in I. 
\end{equation}
Taking the residue of \eqref{DY':xxh} at $z=0$ gives
\begin{equation}
 [x_{i0}^+,x_j^-(w)]=\delta_{ij}\left(\wt h_i^+(w+\tfrac{\hbar \msc}{2})+\wt h_i^-(w)\right), \label{eqn:xi0xj-}
\end{equation}
where we have used that $\delta({z},{w})\wt h_i^-(z)=\delta({z},{w})\wt h_i^-(w)$. The relation \eqref{xxh:suf} follows directly from this identity.

The proof is concluded by noting that the second statement of the proposition is an immediate consequence of the first. \qedhere
\end{proof}
Observe that $\iota_\hbar$ (resp.\ $\iota$) is a graded (resp.\ filtered) homomorphism. 
We conjecture that both $\iota_\hbar$ and $\iota$ are injective. 
%

\section{The Lie algebras \texorpdfstring{$\mfs$}{s} and \texorpdfstring{$\mft$}{t} as central extensions}\label{Sec:aff}

In Sections \ref{Sec:Yg} and \ref{Sec:DYg} it was noted that the Lie algebras $\mfs$ and $\mft$ (see Definitions \ref{D:s} and \ref{D:t}) are always extensions of $\mfg^\prime[t]$ and $\mfg^\prime[t^{\pm 1}]$, respectively. In this section we employ the results of \cite[Prop.\ 3.5]{MRY} to deduce that, when $\mfg$ is of untwisted affine type, $\mfs$ and $\mft$ are in fact isomorphic to the universal central extensions of $\mfg^\prime[t]$ and $\mfg^\prime[t^{\pm 1}]$, respectively.  

Let $\mfg_0$ be the underlying finite-dimensional, simple Lie algebra of the untwisted affine Lie algebra $\mfg$. We specify the indexing set $I$ to  be $\{ 0,1,\ldots,\ell\}$, the extending vertex of the Dynkin diagram of $\mfg$ being labeled by $0$.  Let $\mcA$ be a commutative, associative $\C$-algebra. Then $\mfg_0\ot_{\C}\mcA$ is a Lie algebra in a natural way. Denote by $\Omega^1(\mcA)$ the module of K\"{a}hler differentials of $\mcA$, and let $d\mcA$ denote the subspace of exact forms (see, for instance, \cite[\S2]{MRY}).  
\begin{theorem}[\cite{Kl}, Theorem 3.3]\label{T:Kas}
 The Lie algebra $\mfg_0\ot_{\C}\mcA$ admits a universal central extension $\mathfrak{uce}(\mfg_0\ot_{\C}\mcA)$ defined by
 \begin{equation*}
  \mathfrak{uce}(\mfg_0\ot_{\C}\mcA)=(\mfg_0\ot_{\C}\mcA)\oplus \Omega^1(\mcA)/{d \mcA}
 \end{equation*}
 as a vector space, with Lie bracket such that $\Omega^1(\mcA)/{d \mcA}$ is central and 
 \begin{equation*}
  [X_1 \ot a, X_2 \ot b] 
  = [X_1,X_2] \ot ab + (X_1,X_2)\cdot b(da) \quad \forall \, X_1,X_2\in\mfg_0 \; \text{ and } \; a,b\in\mcA. 
 \end{equation*}
\end{theorem}
We will be interested in the choices $\mcA=\C[t_1^{\pm 1},t_2]$ and $\mcA=\C[t_1^{\pm 1},t_2^{\pm 1}]$. Set 
\begin{equation*}
 \mfg_0[t_1^{\pm 1},t_2]=\mfg_0\otimes_\C \C[t_1^{\pm 1},t_2]\quad \text{ and }\quad \mfg_0[t_1^{\pm 1},t_2^{\pm 1}]=\mfg_0\otimes_\C \C[t_1^{\pm 1},t_2^{\pm 1}].
\end{equation*}
As in \cite[(3.1)]{MRY}, we let $\mft(\mbA)$ denote the Lie algebra obtained from Definition \ref{D:t} by replacing the defining relation
\eqref{t:XX} with
 \begin{equation}\label{[X(z),X(w)]}
  [X_i^\pm(z),X_i^\pm(w)]=0 \quad \forall\; i\in I.
 \end{equation}
It was proven in \cite{MRY} that, in fact,  $\mft(\mbA)\cong \mathfrak{uce}(\mfg_0[t_1^{\pm 1},t_2^{\pm 1}])$. 
 The following lemma asserts that $\mft(\mbA)$ coincides with $\mft$, and hence that $\mft$ can also be identified with $\mathfrak{uce} (\mfg_0[t_1^{\pm 1},t_2^{\pm 1}])$, as will be stated more precisely in Proposition \ref{P:tuce}.
\begin{lemma}\label{L:min}
Assume that $\mfg$ is a symmetrizable Kac-Moody algebra with indecomposable Cartan matrix $\mbA=(a_{ij})_{i,j\in I}$ satisfying the condition \eqref{A-cond}. Then, in the Lie algebra $\mft$, the relation \eqref{t:XX} implies the relation \eqref{[X(z),X(w)]}.
Conversely, the relations \eqref{t:HX}, \eqref{t:X+X-}, \eqref{t:Serre} and \eqref{[X(z),X(w)]} imply that \eqref{t:XX} holds for all $i,j\in I$. In particular, if $\mfg$ is of untwisted affine type (excluding $A_1^{(1)}$), $\mft\cong \mft(\mbA)$.
\end{lemma}
\begin{proof}
 We first prove the implication
$\eqref{t:XX}\implies \eqref{[X(z),X(w)]}$. The relation \eqref{t:XX} with $i=j$ implies that there is $A_{i}(w)\in \mft[\![w^{\pm 1}]\!]$ such that 
\begin{equation*}
 [X_i^\pm(z),X_i^\pm(w)]=\delta({z},{w})A_i(w).
\end{equation*}
Since the right-hand side is symmetric in $w$ and $z$ and the left-hand side is antisymmetric, both sides must be zero, and hence \eqref{[X(z),X(w)]} holds.

To prove that $(\eqref{t:HX},\eqref{t:X+X-}, \eqref{t:Serre},\eqref{[X(z),X(w)]})\implies \eqref{t:XX}$, we make a few preliminary observations.
By taking the residue at $w=0$ of \eqref{t:HX} and then also of the relation obtained from \eqref{t:HX} by exchanging $z$ and $w$, we arrive at the identity
 \begin{equation}\label{[H(z),X]}
  [H_i(z),X_{j0}^\pm]=\pm 2d_{ij}X_j^\pm (z)=[H_{i0},X_j^\pm(z)]\quad  \forall \; i,j\in I. 
 \end{equation}
Similarly, from \eqref{t:X+X-} we obtain
 \begin{equation}\label{[X+,X-]}
  [X_{i0}^\mp,X_j^\pm(w)]=\mp \delta_{ij} H_i(w) \quad \forall \; i,j\in I. 
 \end{equation}
Now fix $i,j\in I$ with $i\neq j$. If $a_{ij}=0$, then \eqref{t:Serre} is the relation $[X_{i0}^\pm,X_j^\pm(w)]=0$. After applying $\mathrm{ad}(H_i(z))$ to this equation and employing \eqref{[H(z),X]} and \eqref{t:HX}, it becomes 
\begin{equation*}
\pm 2d_{ii} [X_{i}^\pm(z),X_j^\pm(w)]\pm 2d_{ij}\delta({z},{w})[X_{i0}^\pm,X_j^\pm(w)]=\pm 2d_{ii} [X_{i}^\pm(z),X_j^\pm(w)]=0,
\end{equation*}
which gives \eqref{t:XX}. 

If $a_{ij}\neq 0$, then without loss of generality we may assume that $a_{ij}=-1$. The Serre relation \eqref{t:Serre} then reads as 
$[X_{i0}^\pm,[X_{i0}^\pm,X_j^\pm(w)]]=0$. Applying $\mathrm{ad}(H_i(z))$ to both sides of this equation, we find that 
\begin{equation*}
 \pm 4d_{ii}[X_{i}^\pm(z),[X_{i0}^\pm,X_j^\pm(w)]]\pm 2d_{ij}\delta({z},{w})[X_{i0}^\pm,[X_{i0}^\pm,X_j^\pm(w)]]=0,
\end{equation*}
where we have used \eqref{t:HX}, \eqref{[X(z),X(w)]} and \eqref{[H(z),X]}. Hence, we have 
\begin{equation*}
 [X_i^\pm(z),[X_{i0}^\pm,X_j^\pm(w)]]=0. 
\end{equation*}
Acting on this identity by $\mathrm{ad}(X_{i0}^\mp)$ and employing \eqref{t:HX} together with \eqref{[H(z),X]}  and \eqref{[X+,X-]}, we deduce that
\begin{equation*}
 2(d_{ii}+d_{ij})[X_i^\pm(z),X_j^\pm(w)]=-2d_{ij}\delta({z},{w})[X_{i0}^\pm,X_j^\pm(w)]. 
\end{equation*}
By assumption, $-1=a_{ij}=2\frac{d_{ij}}{d_{ii}}$, and hence $d_{ii}\neq -d_{ij}$. Multiplying the above equation by 
$(2d_{ii}+2d_{ij})^{-1}(z-w)$ therefore produces the relation \eqref{t:XX}. \qedhere
\end{proof}
\begin{remark}
The generators $X_{ir}^\pm$, $H_{ir}$ and $\mathbf{C}$ of $\mft$ are related to the generators $x_r(\pm\al_i)$, $\al_i^{\vee}(r)$ and $\mathbf{c}$ of $\mft(\mbA)$ given in \cite[(3.1)]{MRY} by 
\begin{equation*}
 X_{ir}^{\pm}=\pm d_{ii}^{-1/2} x_r(\pm\al_i),\quad H_{ir}=d_{ii}^{-1} \al_i^{\vee}(r)\quad \text{ and }\quad \mathbf{C}=\mathbf{c}.
\end{equation*}
\end{remark}
In order to describe the isomorphism $\mft\cong \mathfrak{uce}(\mfg_0[t_1^{\pm 1},t_2^{\pm 1}])$ and its $\mfs$-analogue, 
we will need a more explicit description of $\Omega^1(\mcA)/d\mcA$ when $\mcA=\C[t_1^{\pm 1},t_2^{\pm 1}]$ or $\C[t_1^{\pm 1},t_2]$. By \cite[\S 2]{MRY},  $\Omega^1(\C[t_1^{\pm 1},t_2^{\pm 1}])/d(\C[t_1^{\pm 1},t_2^{\pm 1}])$ has basis 
\begin{equation*}
 B_\mft=\{ t_1^{-1}dt_1, t_1^{k} t_2^{\ell} dt_1, t_1^k t_2^{-1}dt_2 \, :  \, k\in\Z, \ell\in\Z^{\times} \}.
\end{equation*}
Similarly, one finds that $\Omega^1( \C[t_1^{\pm 1},t_2])/d(\C[t_1^{\pm 1},t_2])$ has basis $B_\mfs\subset B_\mft$ given by
\begin{equation*}
 B_\mfs=\{ t_1^{-1}dt_1, t_1^{k} t_2^{\ell} dt_1 \, :  \, k\in\Z, \ell\in\Z_{>0} \}.
\end{equation*}
Note that these observations, coupled with Theorem \ref{T:Kas}, imply that $\mathfrak{uce}(\mfg_0[t_1^{\pm 1},t_2]) \subset \mathfrak{uce}(\mfg_0[t_1^{\pm 1},t_2^{\pm 1}])$ as a Lie subalgebra.  Let $\{ X_i^{\pm}, H_i \}_{i=1}^{\ell}$ be the Chevalley generators for $\mfg_0$ normalized so that $(X_i^+,X_i^{-}) = 1$ and $H_i = [X_i^+,X_i^-]$. Let $X_{\pm \theta}$ be root vectors of $\mfg_0$ for the roots $\pm\theta$ normalized so that $(X_{\theta},X_{-\theta})=1$, where $\theta$ is the highest root of $\mfg_0$. Set $H_{\theta} = [X_{-\theta}, X_{\theta}]$.
\begin{proposition}[Prop.\ 3.5 of \cite{MRY}]\label{P:tuce}
The assignment $\{X_{ir}^\pm,H_{ir},\mathbf{C}\}_{i\in I,r\in \Z} \to \mathfrak{uce}(\mfg_0[t_1^{\pm 1},t_2^{\pm 1}])$ given by
\begin{align*}
\mbC & \mapsto t_1^{-1}dt_1, \\
X_{ir}^{\pm} & \mapsto X_{i}^{\pm} \ot t_1^r, \; i=1,\ldots,\ell, \\
X_{0r}^{\pm} & \mapsto   X_{\mp\theta} \ot t_1^r  t_2^{\pm 1} , \\
H_{ir} & \mapsto H_i \ot t_1^r, \; i=1,\ldots,\ell, \\
H_{0r} & \mapsto   H_{\theta} \ot t_1^r + t_1^r t_2^{-1} dt_2,
\end{align*}
extends to an isomorphism of Lie algebras $\mft\stackrel{\sim}{\rightarrow} \mathfrak{uce}(\mfg_0[t_1^{\pm 1},t_2^{\pm 1}])$. 
Moreover, we have $\mfs\cong \mathfrak{uce}(\mfg_0[t_1^{\pm 1},t_2])$ with an isomorphism 
$\mfs\stackrel{\sim}{\rightarrow} \mathfrak{uce}(\mfg_0[t_1^{\pm 1},t_2])$ given by the above assignment with $r$ taking values in $\Z_{\geq 0}$ and $\mbC$ omitted.
\end{proposition}
\begin{remark}
Although the second part of the above proposition (concerning $\mfs\cong\mathfrak{uce}(\mfg_0[t_1^{\pm 1},t_2])$)  was not stated in \cite[Prop.\ 3.5]{MRY}, it can be proven in the same way as the first part. 
\end{remark}
\begin{corollary}\label{C:s->t}
Assume that $\mfg$ is of untwisted affine type (excluding $A_1^{(1)}$). Then the natural morphism 
\begin{equation*}
 \mfs\to \mft,\quad X_{ir}^\pm \mapsto X_{ir}^\pm, \; H_{ir}\mapsto  H_{ir}\quad \forall \; i\in I\; \text{ and }\;r\geq 0
\end{equation*}
is an embedding of Lie algebras.
\end{corollary}
Due to the following proposition, it is also possible to interpret $\mfs$ and $\mft$ as universal central extensions of $\mfg^{\prime}[t]$ and $\mfg^{\prime}[t^{\pm 1}]$, respectively.
\begin{proposition}\label{P:g'[t]}
We have isomorphisms of Lie algebras 
\begin{equation*}
 \mathfrak{uce}(\mfg_0[t_1^{\pm 1},t_2]) \cong \mathfrak{uce}(\mfg^{\prime}[t]) \quad \text{ and }\quad \mathfrak{uce}(\mfg_0[t_1^{\pm 1},t_2^{\pm 1}]) \cong \mathfrak{uce}(\mfg^{\prime}[t^{\pm 1}]).
\end{equation*}
\end{proposition}
\begin{proof}
We begin by noting that, since $\mfg^{\prime}[t]$ and $\mfg^{\prime}[t^{\pm 1}]$ are perfect Lie algebras because $\mfg^{\prime}$ is perfect, the universal central extensions $\mathfrak{uce}(\mfg^{\prime}[t])$  and $\mathfrak{uce}(\mfg^{\prime}[t^{\pm 1}])$ do in fact exist (see \cite[Thm.\ 1.14]{Ne}). 

Since $\mfg$ is an untwisted affine Lie algebra, $\mfg^\prime \cong \mfg_0[t_1^{\pm 1}]\oplus \C K$ with Lie bracket determined by $[K,\mfg^\prime]=0$ and 
\begin{equation*}
 [X_1\ot t_1^r,X_2\ot t_1^s]=[X_1,X_2]\otimes t_1^{r+s}+r\delta_{r,-s}(X_1,X_2)K  
\end{equation*}
for all $X_1,X_2\in \mfg_0$ and $r,s\in \Z$. It follows that $\mfg^\prime[t_2]\cong \mfg_0[t_1^{\pm 1},t_2]\oplus \C[t_2]K$ is a central extension of $\mfg_0[t_1^{\pm 1},t_2]$ with natural projection $\pi:\mfg^\prime[t_2]\onto \mfg_0[t_1^{\pm 1},t_2]$. 
Let $\psi$ denote the projection $\mathfrak{uce}(\mfg^{\prime}[t_2])\onto \mfg^\prime[t_2]$. Then,
by \cite[Cor.\ 1.9]{Ne}, $\mathfrak{uce}(\mfg^{\prime}[t_2])$ is a universal central extension of $\mfg_0[t_1^{\pm 1},t_2]$ with projection $\pi\circ \psi:\mathfrak{uce}(\mfg^{\prime}[t_2])\onto\mfg_0[t_1^{\pm 1},t_2]$. This proves that  $\mathfrak{uce}(\mfg_0[t_1^{\pm 1},t_2]) \cong \mathfrak{uce}(\mfg^{\prime}[t_2])$. Replacing $t_2$ by $t_2^{\pm 1}$, we obtain instead $\mathfrak{uce}(\mfg_0[t_1^{\pm 1},t_2^{\pm 1}]) \cong \mathfrak{uce}(\mfg^{\prime}[t_2^{\pm 1}])$. \qedhere
\end{proof}

\section{Level one vertex representations}\label{Sec:Ver}

We now fix $\mfg$ to be a simply laced Kac-Moody algebra, and we let $Q=\bigoplus_{i\in I} \Z\al_i$ denote the root lattice associated to $\mfg$. In addition, we normalize the invariant form $(\,,\,)$ so that $(\al_i,\al_i)=2$ for all $i\in I$. 

In this section, we construct representations of $DY^{\msc}_{\hbar}(\mfg)$ and $DY^\msc(\mfg)$ which are given by vertex operators and which factor through $DY_\hbar^1(\mfg)$ and $DY^1(\mfg)$. The main results pertaining to this construction are given in Subsections \ref{ssec:V[[h]]} and \ref{ssec:wtV}.

The vertex operators which define these representations are themselves built from operators arising from the action of a Heisenberg Lie algebra on its Fock space representation. Accordingly, we begin by introducing the appropriate Heisenberg algebra, its Fock space representation, as well as the auxiliary operators which play a central role in our construction.

\begin{definition}\label{D:h}
The Heisenberg algebra $\mfH$ is the Lie algebra with basis given by the elements $\mcH_{ir}, \mcC$ for $i\in I, r\in\Z\setminus\{ 0 \}$ and with the bracket given by
\[ 
[\mcH_{ir}, \mcC]=0, \; \forall \, i\in I, \; \forall \, r\in \Z\setminus\{ 0 \}, 
\quad [\mcH_{ir},\mcH_{j,-s}] =  r \delta_{rs} \delta_{ij} \mcC, \; \forall \, i,j\in I,  \; \forall \, r,s \in \Z\setminus\{ 0 \}.
\]
\end{definition}
\begin{remark} 
This is not the usual definition of the Heisenberg algebra associated to $Q$ (see Definition \ref{D:h_A}): rather, it is the Heisenberg algebra associated to the trivial lattice $\Z^{|I|}$.
\end{remark}
 The polynomial ring $\C[\mcH_{i,-r}]_{i\in I,r>0}$ can be equipped with the structure of an $\mfH$-module by defining
\begin{equation*}
\mcH_{j,-s}(f)=\mcH_{j,-s}f,\quad \mcC(f)=f,\quad \mcH_{js}(f)= s\frac{\partial}{\partial \mcH_{j,-s}}(f) \quad \forall\; f\in  \C[\mcH_{i,-r}]_{i\in I,r>0},\; j\in I\; \text{ and }\; s>0,
\end{equation*}
yielding the so-called \textit{Fock space} representation of $\mfH$.

Next, fix a bimultiplicative function $\veps:Q\times Q\to \Z/2\Z=\{\pm 1\}$ satisfying the condition 
\begin{equation}
 \veps(\alpha,\alpha)=(-1)^{\frac{1}{2}(\alpha,\alpha)} \quad \forall \; \alpha\in Q. \label{coc:1}
\end{equation}
The bimultiplicativity of $\veps$ implies that $\veps(\alpha,0)=1$ for all $\alpha\in Q$, while \eqref{coc:1} implies that 
\begin{equation}
\veps(\alpha,\beta)=(-1)^{(\al,\beta)}\veps(\beta,\alpha) \quad \forall \; \alpha,\beta\in Q. \label{coc:2}
\end{equation}
Using $\veps(0,\beta)=1=\veps(\alpha,0)$, one can also see that
\begin{equation}
\veps(\pm \al,\mp \beta)=\veps(\al,\beta)=\veps(\pm \al,\pm \beta)\quad \forall \; \alpha,\beta\in Q. \label{coc:3}
\end{equation}
The bimultiplicativity of $\veps$ also implies that it is a 2-cocycle of $Q$ with values in $\Z/2\Z$, and thus it determines a central extension $\wt Q=\Z/2\Z\times_\veps Q$ of $Q$ by $\Z/2\Z$ which is equal to $\Z/2\Z \times Q$ as a set, and has product  
\begin{equation*}
 (\eps_a,\alpha)(\eps_b,\beta)=(\veps(\alpha,\beta)\eps_a\eps_b,\alpha+\beta) \quad \forall \alpha,\beta\in Q \; \text{ and }\; \eps_a,\eps_b\in \Z/2\Z.
\end{equation*}
\begin{definition}\label{def:Ce[Q]}
Let $\mcI$ be the two-sided ideal of the group algebra $\C[\wt Q]$ which is spanned by $e^{(\eps_a,\alpha)} -  \eps_a e^{(1,\alpha)}$ for all $\al\in Q$ and $\eps_a\in \Z/2\Z$, where $\{e^{(\epsilon_a,\alpha)}\,:\,(\epsilon_a,\alpha)\in \wt Q\}$ is the standard basis of $\C[\wt Q]$.   The twisted group algebra $\C_\veps[Q]$ is defined to be the quotient $\C[\wt Q]/\mcI$. 
\end{definition}
Since the $\C$-linear projection $\C[\wt Q]\onto \C[Q]$, $e^{(\eps_a,\alpha)}\mapsto \eps_a e^\alpha$ induces an isomorphism of vector spaces $\C_\veps[Q]\to \C[Q]$, $\C_\veps[Q]$ can be equivalently defined as the $\C$-algebra with basis $\{e^\alpha\}_{\alpha\in Q}$ and multiplication given by 
 \begin{equation*}
  e^\alpha\cdot e^\beta=\veps(\alpha,\beta)e^{\alpha+\beta} \quad \forall \; \al,\beta\in Q. 
 \end{equation*}
\begin{remark}\leavevmode
\begin{enumerate}[label=(\alph*)]
\item By \cite[Prop.\ 5.2.3]{FLM}, the condition \eqref{coc:2} determines $\veps$ up to equivalence of cocycles, and hence it determines the central extension $\wt Q$ of $Q$ by $\Z/2\Z$ up to isomorphism. In particular, this implies that any two bimultiplicative functions $\veps,\veps^\prime$ satisfying \eqref{coc:1} will determine the same twisted group algebra up to isomorphism. 
\item The existence of $\veps:Q\times Q\to \Z/2\Z$ satisfying \eqref{coc:1} can be established in various ways: see for instance \cite[\S 7.8]{Ka}. 
\end{enumerate}
\end{remark}
Viewing $\C_\veps[Q]$ as a left-module over itself, we can form the $U(\mfH)\ot \C_\eps[Q]$-module 
\begin{equation}
 \mcV=\C[\mcH_{i,-r}]_{i\in I,r>0}\ot \C_\veps[Q] \label{V}.
\end{equation}
 We also define an auxiliary family of operators $\{\partial_\alpha\}_{\alpha\in Q}\subset \End_\C\mcV$ by
\begin{equation*}
\partial_\alpha(f\ot e^\beta)=(\alpha,\beta) f\ot e^\beta \quad \forall \; f\in \C[\mcH_{i,-r}]_{i\in I,r>0}\; \text{ and }\; \alpha,\beta\in Q.  
\end{equation*}
%


\subsection{The \texorpdfstring{$DY_\hbar^\msc(\mfg)$}{}-module \texorpdfstring{$\mcV[\![\hbar]\!]$}{}}\label{ssec:V[[h]]}

We first construct a vertex representation of $DY_\hbar^\msc(\mfg)$ on the topologically free $\C[\![\hbar]\!]$-module $\mcV[\![\hbar]\!]$. The actions of $U(\mfH)\ot \C_\veps[Q]$ and of $\partial_{\alpha}$ on $\mcV$ defined above naturally extend to $\mcV[\![\hbar]\!]$. 

For each $i\in I$, let $N(i)$ denote the set of vertices to which $i$ is connected, i.e. the set of \textit{neighbours} of the vertex $i$. Define $A_i^\pm(z)$ and $B_i^\pm(z)$, for each $i\in I$, by 
\begin{gather*}
 A_i^\pm(z)=\exp\!\left(\pm \sum_{r>0}\frac{\mcH_{i,-r}}{r}(z^r+(z\mp\hbar)^r)\mp \sum_{r>0}\sum_{j\in N(i)} \frac{\mcH_{j,-r}}{r}\left(z\mp\tfrac{\hbar}{2}\right)^r \right),\\
 B_i^\pm(z)=\exp\!\left( \mp \sum_{r>0}\frac{\mcH_{ir}}{r}z^{-r}\right).
\end{gather*}
Inspired by \cite{Io}, we define the vertex operators $\mbX_i^\pm(z),\mbH_i^\pm(z)\in (\End_{\C[\![\hbar]\!]}\mcV[\![\hbar]\!])[\![z^{\pm 1}]\!]$, for each $i\in I$, by 
\begin{align}
\mbX_i^\pm(z)= {} & \pm A_i^\pm(z)B_i^\pm(z)e^{\pm \al_i}z^{\partial_{\pm \al_i}}, \label{Xpm:exp} \\
\mbH_i^+(z)= {} & B_i^+(z+\tfrac{\hbar}{2})B_i^-(z-\tfrac{\hbar}{2})\left(\frac{1+\frac{\hbar}{2}z^{-1}}{1-\frac{\hbar}{2}z^{-1}} \right)^{\partial_{\al_i}}, \label{Hp:exp}
\\
\mbH_i^-(z)= {} & A_i^+(z)A_i^-(z). \label{Hm:exp}
\end{align}
where, for each $\alpha\in Q$, $z^{\partial_\al}\in (\End_{\C[\![\hbar]\!]}\mcV[\![\hbar]\!])[\![z^{\pm 1}]\!]$ is defined on $\mcV$ by 
\[ z^{\partial_\al}(f\ot e^\beta)=z^{(\al,\beta)}f\ot e^\beta \quad \forall\; \beta\in Q \;\text{ and }\; f\in \C[\mcH_{i,-r}]_{i\in I,r>0}. \]

Equivalently, $z^{\partial_{\al}}=\sum_{k\in \Z} P_{\al,k} z^k$ with $P_{\al,k}(f\ot e^\beta)=\delta_{k,(\al,\beta)}f\ot e^\beta$.

Let us explain why $\left(\tfrac{1+\frac{\hbar}{2}z^{-1}}{1-\frac{\hbar}{2}z^{-1}}\right)^{\partial_{\al_i}}$, and thus $\mbH_i^+(z)$, belongs to $(\End_{\C[\![\hbar]\!]}\mcV[\![\hbar]\!])[\![z^{-1}]\!]$. For each invertible series $g(z)\in (\C[\hbar])[\![z^{-1}]\!]$, the operator $g(z)^{\partial_\al}$ (defined on $\mcV$ by $g(z)^{\partial_\al}(f\ot e^\beta)=g(z)^{(\al,\beta)}f\ot e^\beta$) can be viewed as an element of $(\End_{\C[\![\hbar]\!]}\mcV[\![\hbar]\!])[\![z^{-1}]\!]$. To see this, first write
\begin{equation*}
 \sum_{k\in \Z} P_{\al,k}g(z)^k=\sum_{k\geq 0} P_{\al,k} g(z)^k +\sum_{k>0}P_{\al,-k}(g(z)^{-1})^k.
\end{equation*}
For each $r\geq 0$, the $z^{-r}$ coefficient of $\sum_{k\geq 0} P_{\al,k} g(z)^k$ is an infinite sum of the form $\sum_{k\geq 0}a_k(\hbar)P_{\al,k}$ with 
$a_k(\hbar)\in \C[\hbar]$. The sum  $\sum_{k\geq 0}a_k(\hbar)P_{\al,k}$ is a well-defined element of $\End_{\C[\![\hbar]\!]}\mcV[\![\hbar]\!]$ since, for any fixed $\beta\in Q$, $P_{\al,k}(e^\beta)=0$ for all but at most one value of $k$. This implies that $\sum_{k\geq 0} P_{\al,k} g(z)^k\in (\End_{\C[\![\hbar]\!]}\mcV[\![\hbar]\!])[\![z^{-1}]\!]$. The same reasoning can be applied to $\sum_{k>0}P_{\al,-k}(g(z)^{-1})^k$.

With the vertex operators $\{\mbX_i^\pm(z), \mbH_i^\pm(z)\}_{i\in I}$ at our disposal, we can now state the main theorem of this section. 
\begin{theorem} \label{thm:Vrep}
The assignment 
\begin{equation}
 x_i^\pm(z) \mapsto \mbX_i^\pm(z), \; h_i^\pm(z) \mapsto \mbH_i^\pm(z) \quad \forall \; i\in I, \;  \msc \mapsto 1 \label{Vrep}
\end{equation}
extends to a homomorphism of $\C[\hbar]$-algebras $\rho_\hbar:DY_\hbar^\msc(\mfg)\to \End_{\C[\![\hbar]\!]}\mcV[\![\hbar]\!]$.
\end{theorem} 
The next lemma will be employed to prove this theorem. Let \begin{equation*}
 \Gamma_i^\pm(z)=\exp\left(\mp \sum_{r>0}\frac{\mcH_{i,\pm r}}{r}z^{\mp r}\right). 
\end{equation*}
\begin{lemma}\label{L:AB}
Let $\chi_i:I\to \{0,1\}$ denote the indicator function of $N(i)$, i.e. $\chi_i(j)=1$ if $j\in N(i)$ and $\chi_i(j)=0$ otherwise. Then, for each pair of indices $i,j\in I$, we have
\begin{gather}
 \Gamma_i^-(z)\Gamma_i^+(w)=\Gamma_i^+(w)\Gamma_i^-(z)\left(1-\frac{z}{w}\right)^{-1},  \label{Comm:1} \\
 [A_i^\pm(z),A_j^\pm(w)]=[A_i^\pm(z),A_j^\mp(w)]=0=[B_i^\pm(z),B_j^\mp(w)]=[B_i^\pm(z),B_j^\pm(w)],\label{AABB} \\
 B_i^\pm(z)A_j^\pm(w)=\frac{(1-z^{-1}w)^{\delta_{ij}}(1-z^{-1}(w\mp \hbar))^{\delta_{ij}}}{(1-z^{-1}(w\mp\tfrac{\hbar}{2}))^{\chi_j(i)}}A_j^\pm(w)B_i^\pm(z),\label{BA:1}\\
 B_i^\pm(z)A_j^\mp(w)=\frac{\left(1-z^{-1}(w\pm \tfrac{\hbar}{2}) \right)^{\chi_j(i)}}{(1-z^{-1}w)^{\delta_{ij}}(1-z^{-1}(w\pm \hbar))^{\delta_{ij}}}A_j^\mp(w)B_i^\pm(z). \label{BA:2}
\end{gather}
\end{lemma}
\begin{proof}
Relations of the form \eqref{Comm:1} appear often in the literature: see for instance the proof of Theorem 14.8 in \cite{Ka} and the proof of Proposition 2.9(a) in \cite{FrKa}. It follows from the fact that 
\begin{equation*}
 \exp(A)\exp(B)=\exp(B)\exp(A)\exp([A,B])
\end{equation*}
for any two operators $A$ and $B$ such that $[A,[A,B]]=0=[B,[A,B]]$, together with the relation
\begin{equation*}
\left[ \sum_{r>0}  \frac{\mcH_{i,-r}}{r} z^r ,\sum_{s>0} \frac{\mcH_{j,s}}{s} w^{-s}\right] =  \sum_{r,s>0}\frac{[\mcH_{i,-r},\mcH_{j,s}]}{rs} z^r w^{-s} = -\del_{ij} \sum_{s>0}  \frac{1}{s}\left(\frac{z}{w}\right)^s = \del_{ij}\ln\! \left( 1 - \frac{z}{w} \right). 
\end{equation*}
 The relation \eqref{AABB} is immediate from the definition of the operators $A_i^\pm(z)$ and $B_i^\pm(z)$, while \eqref{BA:1} is a straightforward application of \eqref{Comm:1}. The relation \eqref{BA:2} is a consequence of \eqref{BA:1} since $B_i^\pm(z)=B_i^\mp(z)^{-1}$. 
\end{proof}
We will also need the following identity, which can be deduced immediately from the definition of $g(z)^{\partial_\al}$:
\begin{equation*}
   g(z)^{\partial_\al}e^\beta=e^\beta g(z)^{(\al,\beta)} g(z)^{\partial_\al} \quad \forall\; \alpha,\beta\in Q \; \text{ and }\; g(z) = z \text{ or } g(z)\in ((\C[\hbar])[\![z^{-1}]\!])^\times. \label{Comm:2}
\end{equation*}
\begin{proof}[Proof of Theorem \ref{thm:Vrep}]
The proof is achieved using standard vertex operator calculus.  We will prove that the relations of Definition \ref{DY:def} are preserved by the assignment \eqref{Vrep}. This is immediate for \eqref{DY:c}, and for \eqref{DY:hh} this is a consequence of the relation \eqref{AABB} of Lemma \ref{L:AB} and that $\mcV[\![\hbar]\!]$ is torsion free. The other relations require more elaborate use of Lemma \ref{L:AB} and we will treat them independently. Set 
\begin{equation*}
 z_\pm =z\pm \tfrac{\hbar}{2}.
\end{equation*}
%
%
\begin{proof}[The relation \eqref{DY:hhh}]\let\qed\relax Let $c_{ij}^\pm=\tfrac{\hbar}{2}((\al_i,\al_j)\pm 1)$ denote the image of $\mathsf{c}_{ij}^\pm$ under \eqref{Vrep}. Then
\begin{align*}
 \mbH_i^+(z) \mbH_j^-(w) &=B_i^+(z_+)B_i^-(z_-)A_j^+(w)A_j^-(w)\left(\frac{z_+}{z_-}\right)^{\partial_{\al_i}}\\
 &=\frac{\left(1-z_-^{-1}w_- \right)^{\chi_j(i)}}{(1-z_-^{-1}w)^{\delta_{ij}}(1-z_-^{-1}(w-\hbar))^{\delta_{ij}}}B_i^+(z_+)A_j^+(w)B_i^-(z_-)A_j^-(w)\left(\frac{z_+}{z_-}\right)^{\partial_{\al_i}}\\
 &=\frac{\left(1-z_-^{-1}w_- \right)^{\chi_j(i)}(1-z_+^{-1}w)^{\delta_{ij}}(1-z_+^{-1}(w - \hbar))^{\delta_{ij}}}{(1-z_-^{-1}w)^{\delta_{ij}}(1-z_-^{-1}(w - \hbar))^{\delta_{ij}}(1-z_+^{-1}w_-)^{\chi_j(i)}}A_j^+(w)B_i^+(z_+)B_i^-(z_-)A_j^-(w)\left(\frac{z_+}{z_-}\right)^{\partial_{\al_i}}.
\end{align*}
On the other hand, we have 
\begin{align*}
 \mbH_j^-(w)\mbH_i^+(z) &=A_j^+(w)A_j^-(w)B_i^+(z_+)B_i^-(z_-)\left(\frac{z_+}{z_-}\right)^{\partial_{\al_i}}\\
 &=\frac{(1-z_+^{-1}w)^{\delta_{ij}}(1-z_+^{-1}(w+\hbar))^{\delta_{ij}}}{(1-z_+^{-1}w_+)^{\chi_j(i)}}A_j^+(w)B_i^+(z_+)A_j^-(w)B_i^-(z_-)\left(\frac{z_+}{z_-}\right)^{\partial_{\al_i}}\\
 &=\frac{(1-z_+^{-1}w)^{\delta_{ij}}(1-z_+^{-1}(w+\hbar))^{\delta_{ij}}(1-z_-^{-1}w_+)^{\chi_j(i)}}{(1-z_+^{-1}w_+)^{\chi_j(i)}(1-z_-^{-1}w)^{\delta_{ij}}(1-z_-^{-1}(w + \hbar))^{\delta_{ij}}}
    A_j^+(w)B_i^+(z_+)B_i^-(z_-)A_j^-(w)\left(\frac{z_+}{z_-}\right)^{\partial_{\al_i}}.
\end{align*}
Therefore, since $\mcV[\![\hbar]\!]$ is torsion free, the assignment \eqref{Vrep} will preserve \eqref{DY:hhh} provided 
\begin{gather*}
 ((z-w)^2-(c_{ij}^-)^2)\frac{\left(1-z_-^{-1}w_- \right)^{\chi_j(i)}(1-z_+^{-1}w)^{\delta_{ij}}(1-z_+^{-1}(w-\hbar))^{\delta_{ij}}}{(1-z_-^{-1}w)^{\delta_{ij}}(1-z_-^{-1}(w - \hbar))^{\delta_{ij}}(1-z_+^{-1}w_-)^{\chi_j(i)}}\\
 =((z-w)^2-(c_{ij}^+)^2)\frac{(1-z_+^{-1}w)^{\delta_{ij}}(1-z_+^{-1}(w+\hbar))^{\delta_{ij}}(1-z_-^{-1}w_+)^{\chi_j(i)}}{(1-z_+^{-1}w_+)^{\chi_j(i)}(1-z_-^{-1}w)^{\delta_{ij}}(1-z_-^{-1}(w+\hbar))^{\delta_{ij}}}.
\end{gather*}
for all $i,j\in I$. This can be checked directly using 
\begin{equation}
c_{ij}^\pm=\begin{cases}
            \pm\tfrac{\hbar}{2} &\text{ if }i\neq j,\,\chi_j(i)=0,\\
            -\tfrac{\hbar}{2}\pm \frac{\hbar}{2} &\text{ if } i\neq j, \, \chi_j(i)=1,\\
            \hbar \pm \tfrac{\hbar}{2} &\text{ if } i=j.
           \end{cases}\label{cijpm} 
\end{equation}
\end{proof}

\begin{proof}[The relation \eqref{DY:xh+}]\let\qed\relax Making use of \eqref{Comm:2} together with Lemma \ref{L:AB}, we deduce that
 \begin{align*}
 \mbH_i^+(z) & \mbX_j^\pm(w)\\
 &=\pm B_i^+(z_+)B_i^-(z_-)A_j^\pm(w)B_j^\pm(w)\left(\frac{z_+}{z_-} \right)^{\partial_{\al_i}}e^{\pm \al_j}w^{\partial_{\pm \al_j}}\\
 &=\pm\frac{(1-z_-^{-1}w_\mp)^{\pm \chi_j(i)}\left(\tfrac{z_+}{z_-} \right)^{\pm(\al_i,\al_j)}}{(1-z_-^{-1}w)^{\pm \delta_{ij}}(1-z_-^{-1}(w\mp \hbar))^{\pm \delta_{ij}}} 
                         B_i^+(z_+)A_j^\pm(w)B_i^-(z_-)B_j^\pm(w)e^{\pm \al_j}w^{\partial_{\pm \al_j}}\left(\frac{z_+}{z_-} \right)^{\partial_{\al_i}}\\
 &=\frac{(1-z_-^{-1}w_\mp)^{\pm \chi_j(i)}(1-z_+^{-1}w)^{\pm \delta_{ij}}(1-z_+^{-1}(w\mp \hbar))^{\pm \delta_{ij}}}{(1-z_-^{-1}w)^{\pm \delta_{ij}}(1-z_-^{-1}(w\mp \hbar))^{\pm \delta_{ij}}(1-z_+^{-1}w_\mp)^{\pm \chi_j(i)}} \left(\frac{z_+}{z_-} \right)^{\pm(\al_i,\al_j)}
                         \mbX_j^\pm(w)\mbH_i^+(z).
 \end{align*}
 Therefore, the assignment \eqref{Vrep} will preserve the relation \eqref{DY:xh+} if the following identity holds: 
 \begin{equation*}
  (z-w\pm c_{ij}^+)
  =(z-w\mp c_{ij}^-)\frac{(1-z_-^{-1}w_\mp)^{\pm \chi_j(i)}(1-z_+^{-1}w)^{\pm \delta_{ij}}(1-z_+^{-1}(w\mp \hbar))^{\pm \delta_{ij}}}{(1-z_-^{-1}w)^{\pm \delta_{ij}}(1-z_-^{-1}(w\mp \hbar))^{\pm \delta_{ij}}(1-z_+^{-1}w_\mp)^{\pm \chi_j(i)}} \left(\frac{z_+}{z_-} \right)^{\pm(\al_i,\al_j)}.\\ 
 \end{equation*}
This is easily verified using \eqref{cijpm}. If $i\neq j$ and $\chi_j(i)=0$ then this is clear. If $i\neq j$ and $\chi_j(i)=1$, then the right-hand side equals
 \begin{equation*}
  (z-w\pm\hbar)\left(\frac{1-z_-^{-1}w_\mp}{1-z_+^{-1}w_\mp}\right)^{\pm 1}\left(\frac{z_-}{z_+}\right)^{\pm 1}=(z-w\pm\hbar)\left(\frac{z-w-\tfrac{\hbar}{2}\pm \tfrac{\hbar}{2}}{z-w+\tfrac{\hbar}{2}\pm \tfrac{\hbar}{2}}\right)^{\pm 1}=z-w,
 \end{equation*}
which is the left-hand side. If $i=j$, then $c_{ij}^+=\tfrac{3\hbar}{2}$ and $c_{ij}^-=\tfrac{\hbar}{2}$, and the right-hand side of the equality is 
\begin{equation*}
(z-w\mp \tfrac{\hbar}{2})\left(\frac{(z_+-w)(z_+-w\pm \hbar)}{(z_--w\pm \hbar)(z_--w)} \right)^{\pm 1}=(z-w\mp \tfrac{\hbar}{2})\left(\frac{(z-w+\tfrac{\hbar}{2})(z-w+\tfrac{\hbar}{2}\pm \hbar)}{(z-w-\tfrac{\hbar}{2}\pm \hbar)(z-w-\tfrac{\hbar}{2})} \right)^{\pm 1}=z-w\pm\tfrac{3\hbar}{2}.
\end{equation*}
Note that in neglecting the factor of $\hbar^{-1}$ which appears in \eqref{DY:xh+}, we have  made use of the fact that $\mcV[\![\hbar]\!]$ is torsion free.
\end{proof}
%
%
\begin{proof}[The relation \eqref{DY:xh-}]\let\qed\relax Applying again the relations of Lemma \ref{L:AB}, we obtain 

\begin{align*}
\mbH_i^-(z)\mbX_j^\pm(w)&= \pm A_j^\pm(w)A_i^+(z)A_i^-(z)B_j^\pm(w)e^{\pm \al_j}w^{\partial_{\pm \al_j}}\\
                        &=\pm \frac{(1-w^{-1}z)^{\pm \delta_{ij}}(1-w^{-1}(z+\hbar))^{\pm \delta_{ij}}}{(1-w^{-1}z_+)^{\pm \chi_j(i)}}A_j^\pm(w)A_i^+(z)B_j^\pm(w)A_i^-(z)e^{\pm \al_j}w^{\partial_{\pm \al_j}}\\
                        &=\frac{(1-w^{-1}(z+\hbar))^{\pm \delta_{ij}}(1-w^{-1}z_-)^{\pm \chi_j(i)}}{(1-w^{-1}z_+)^{\pm \chi_j(i)}(1-w^{-1}(z-\hbar))^{\pm \delta_{ij}}}\mbX_j^\pm(w)\mbH_i^-(z)\\
                        &=\left(\frac{1-w^{-1}\left(z\pm \hbar d_{ij} \right)}{1-w^{-1}\left(z\mp\hbar d_{ij}\right)}\right)\mbX_j^\pm(w)\mbH_i^-(z). 
\end{align*}
(The last equality is obtained by considering the three cases $i=j$ and, when $i\neq j$, $\chi_j(i)=1$ and $\chi_j(i)=0$.) Multiplying both sides by $\left(w-z\pm \hbar d_{ij}\right)$ and using the fact that $\mcV[\![\hbar]\!]$ is torsion free, we find that the relation \eqref{DY:xh-} is preserved by the assignment \eqref{Vrep}. 
 
 \begin{NB}
 If $\chi_j(i)=1$, then 
 \begin{equation*}
  \frac{(1-w^{-1}(z+\hbar))^{\pm \delta_{ij}}(1-w^{-1}z_-)^{\pm \chi_j(i)}}{(1-w^{-1}z_+)^{\pm \chi_j(i)}(1-w^{-1}(z-\hbar))^{\pm \delta_{ij}}}
   =\left(\frac{w-z+\tfrac{\hbar}{2}}{w-z-\tfrac{\hbar}{2}} \right)^{\pm 1}=\left(\frac{w-z\pm \tfrac{\hbar}{2}}{w-z\mp \tfrac{\hbar}{2}}\right).
 \end{equation*}
If $i=j$, we instead get 
\begin{equation*}
 \frac{(1-w^{-1}(z+\hbar))^{\pm \delta_{ij}}(1-w^{-1}z_-)^{\pm \chi_j(i)}}{(1-w^{-1}z_+)^{\pm \chi_j(i)}(1-w^{-1}(z-\hbar))^{\pm \delta_{ij}}}
   =\left(\frac{w-z-\hbar}{w-z+\hbar}\right)^{\pm 1}=\left(\frac{w-z\mp \hbar}{w-z\pm \hbar}\right).
\end{equation*}
 \end{NB}

\end{proof}
%
%
To prove that the remaining three relations of Definition \ref{DY:def} are satisfied by $\{\mbX_i^\pm(z),\mbH_i^\pm(z)\}_{i\in I}$, we introduce the following normal ordering: given a finite integer $n\in \Z_{>0}$  together with collections $i_a\in I$ and $\eps_a\in \{\pm\}$ for each $1\leq a\leq n$, set 
\begin{equation*}
 :\!\mbX_{i_1}^{\eps_1}(z_1)\cdots \mbX_{i_n}^{\eps_n}(z_n)\!:= \left(\prod_{b=1}^n \eps_b \right)A_{i_1}^{\eps_1}(z_1)\cdots A_{i_n}^{\eps_n}(z_n)B_{i_1}^{\eps_1}(z_1)\cdots B_{i_n}^{\eps_n}(z_n)e^{\eps_1 \al_{i_1}+\cdots+\eps_n \al_{i_n}}z_1^{\partial_{\eps_1 \al_{i_1}}}\cdots z_n^{\partial_{\eps_n \al_{i_n}}}.
\end{equation*}
Note that with this definition, \eqref{AABB} implies that  $:\!\mbX_{i_1}^{\eps_1}(z_1)\cdots \mbX_{i_n}^{\eps_n}(z_n)\!:\, =\, :\!\mbX_{i_{\sigma(1)}}^{\eps_{\sigma(1)}}(z_{\sigma(1)})\cdots \mbX_{i_{\sigma(n)}}^{\eps_{\sigma(n)}}(z_{\sigma(n)})\!:$ for each permutation $\sigma\in S_n$.

By \eqref{Comm:2} and Lemma \ref{L:AB}, we have 
\begin{align}
 \mbX_i^\pm(z)\mbX_j^\mp(w)&=\veps(\al_i,\al_j)\frac{(1-z^{-1}w_\pm)^{\chi_j(i)}}{(1-z^{-1}w)^{\delta_{ij}}(1-z^{-1}(w\pm \hbar))^{\delta_{ij}}z^{(\al_i,\al_j)}}:\!\mbX_i^\pm(z)\mbX_j^\mp(w)\!:,
 \label{NO:1}\\
 \mbX_i^\pm(z)\mbX_j^\pm(w)&= \veps(\al_i,\al_j)\frac{(1-z^{-1}w)^{\delta_{ij}}(1-z^{-1}(w\mp \hbar))^{\delta_{ij}}}{(1-z^{-1}w_\mp)^{\chi_j(i)}z^{-(\al_i,\al_j)}}:\!\mbX_i^\pm(z)\mbX_j^\pm(w)\!:,
 \label{NO:2}
\end{align}
where we have used \eqref{coc:3}. 
%
%
\begin{proof}[The relation \eqref{DY:xx}]\let\qed\relax
 By \eqref{NO:2}, the equality 
 \begin{equation*}
 (z-w\mp \hbar d_{ij})\mbX_i^\pm(z)\mbX_j^\pm(w)=(z-w\pm \hbar d_{ij})\mbX_j^\pm(w)\mbX_i^\pm(z)\quad \forall \; i,j\in I 
 \end{equation*}
will be satisfied provided the following identity holds: 
\begin{equation*}
 (z-w\mp \hbar d_{ij})\frac{(1-z^{-1}w)^{\delta_{ij}}(1-z^{-1}(w\mp \hbar))^{\delta_{ij}}}{(1-z^{-1}w_\mp)^{\chi_j(i)}z^{-(\al_i,\al_j)}}
 =(-1)^{(\al_i,\al_j)}(z-w\pm \hbar d_{ij})\frac{(1-w^{-1}z)^{\delta_{ij}}(1-w^{-1}(z\mp \hbar))^{\delta_{ij}}}{(1-w^{-1}z_\mp)^{\chi_j(i)}w^{-(\al_i,\al_j)}}.
\end{equation*}
Using that $(\al_i,\al_j)=2\delta_{ij}-\chi_j(i)$, we may rewrite this as 
\begin{equation*}
 (z-w\mp \delta_{ij}\hbar \pm \tfrac{\hbar\cdot \chi_j(i)}{2})\frac{(z-w)^{\delta_{ij}}(z-w\pm \hbar)^{\delta_{ij}}}{(z-w\pm \frac{\hbar}{2})^{\chi_j(i)}}
 =(-1)^{\chi_j(i)}(z-w\pm \delta_{ij}\hbar\mp \tfrac{\hbar\cdot \chi_j(i)}{2})\frac{(w-z)^{\delta_{ij}}(w-z\pm \hbar)^{\delta_{ij}}}{(w-z \pm \frac{\hbar}{2})^{\chi_j(i)}}.
\end{equation*}
If $i\neq j$, then either $\chi_j(i)=1$ and both sides are equal to $1$, or $\chi_j(i)=0$ and both sides equal $(z-w)$. If instead $i=j$, then both sides are equal to the polynomial 
\begin{equation*}
(z-w\pm \hbar)(z-w\mp \hbar)(z-w). \qedhere
\end{equation*}
\end{proof}
To prove that the relations \eqref{DY:xxh} and \eqref{DY:serre} are preserved by \eqref{Vrep}, we employ the following well-known property of the formal delta function $\delta({z},{w})$ which can be found in \cite[Lem.\ 7.7]{Ka} and \cite[Prop.\ 2.1.8 (b)]{LeLi}:
given a vector space $V$ and $f(z,w)\in V[\![z^{\pm 1},w^{\pm 1}]\!]$, we have
 \begin{equation}
  f(z,w)\cdot \delta({z},{w})=f(z,z)\cdot \delta({z},{w}), \label{delta}
 \end{equation}
 provided both sides of this equality are well-defined elements of $V[\![z^{\pm 1},w^{\pm 1}]\!]$. 
%
%
\begin{proof}[The relation \eqref{DY:xxh}]\let\qed\relax From \eqref{NO:1} we obtain the equality of operators 
\begin{equation}
 [\mbX_i^+(z),\mbX_j^-(w)]= \veps(\al_i,\al_j) F_{i,j}(z,w) :\! \mbX_i^+(z)\mbX_j^-(w)\!:, \label{x+x-}
\end{equation}
where $F_{i,j}(z,w)$ is given by 
\begin{equation}
 F_{i,j}(z,w)=\left(\frac{(1-z^{-1}w_+)^{\chi_j(i)}z^{-(\al_i,\al_j)}}{(1-z^{-1}w)^{\delta_{ij}}(1-z^{-1}(w+\hbar))^{\delta_{ij}}} 
                                  -
                                  (-1)^{(\al_i,\al_j)}\frac{(1-w^{-1}z_-)^{\chi_i(j)}w^{-(\al_i,\al_j)}}{(1-w^{-1}z)^{\delta_{ij}}(1-w^{-1}(z- \hbar))^{\delta_{ij}}}
 \right).
\end{equation}
If $i\neq j$ and $\chi_j(i)=0$, then it is clear that $F_{i,j}(z,w)=0$. If $i\neq j$ and $\chi_j(i)=1$, then we again obtain 
\begin{equation*}
 F_{i,j}(z,w)=(1-z^{-1}w_+)z+(1-w^{-1}z_-)w=0.
\end{equation*}
Hence, we have shown that the assignment \eqref{Vrep} preserves the relation \eqref{DY:xxh} when $i\neq j$. If $i=j$, we have 
\begin{align*}
 F_{i,i}(z,w)&=\frac{z^{-2}}{(1-z^{-1}w)(1-z^{-1}(w+\hbar))}-\frac{w^{-2}}{(1-w^{-1}z)(1-w^{-1}(z-\hbar))}\\
             &=\frac{z^{-1}}{1-z^{-1}w}\left(\frac{z^{-1}}{1-z^{-1}(w+\hbar)}+\frac{(w+\hbar)^{-1}}{1-(w+\hbar)^{-1}z}\right)\\
             &\qquad -\frac{z^{-1}}{1-z^{-1}w}\frac{(w+\hbar)^{-1}}{1-(w+\hbar)^{-1}z}-\frac{w^{-2}}{(1-w^{-1}z)(1-w^{-1}(z-\hbar))}\\
             &=\frac{z^{-1}}{1-z^{-1}w}\delta({w+\hbar},{z})-\frac{w^{-1}}{1-w^{-1}(z-\hbar)}\delta({w},{z}),
\end{align*}
where we have used the identities \eqref{eq:deltaexp}  and $\frac{(w+\hbar)^{-1}}{1-(w+\hbar)^{-1}z}=\frac{w^{-1}}{1-w^{-1}(z-\hbar)}$. Substituting the above expression for $F_{i,i}(z,w)$ into \eqref{x+x-} and using that $\veps(\al_i,\al_i)=-1$, we obtain 
\begin{equation}\label{x+x-'}
 [\mbX_i^+(z),\mbX_i^-(w)]=-\delta({w+\hbar},{z})\tfrac{z^{-1}}{1-z^{-1}w}:\! \mbX_i^+(z)\mbX_i^-(w)\!:+ \;\delta({w},{z})\tfrac{w^{-1}}{1-w^{-1}(z-\hbar)}:\! \mbX_i^+(z)\mbX_i^-(w)\!:.
\end{equation}
By \eqref{delta}, 
\begin{align*}
 \delta(w+\hbar,z)\tfrac{z^{-1}}{1-z^{-1}w}:\! \mbX_i^+(z)\mbX_i^-(w)\!:&=\delta(w+\hbar,z)\tfrac{(w+\hbar)^{-1}}{1-(w+\hbar)^{-1}w}:\! \mbX_i^+(z)\mbX_i^-(w)\!:|_{z\mapsto w+\hbar}\\
 &=-\tfrac{1}{\hbar}\delta(w+\hbar,z)A_i^+(w+\hbar)A_i^-(w)B_i^+(w+\hbar)B_i^-(w)\left(\tfrac{w+\hbar}{w}\right)^{\partial_{\al_i}}\\
 &=-\tfrac{1}{\hbar}\delta(w+\hbar,z)\mbH_i^+(w+\tfrac{\hbar}{2}),
\end{align*}
since $A_i^+(w+\hbar)=A_i^-(w)^{-1}$ (see \eqref{Hp:exp}). Similarly, \eqref{delta} implies that 
\begin{align*}
 \delta({w},{z})\tfrac{w^{-1}}{1-w^{-1}(z-\hbar)}:\! \mbX_i^+(z)\mbX_i^-(w)\!:&=\tfrac{1}{\hbar}\delta({w},{z}):\! \mbX_i^+(z)\mbX_i^-(w)\!:|_{w\to z}\\
                 &=-\tfrac{1}{\hbar}\delta({w},{z})A_i^+(z)A_i^-(z)B_i^+(z)B_i^-(z)=-\tfrac{1}{\hbar}\delta({w},{z})\mbH_i^-(z),
\end{align*}
where we have used \eqref{Hm:exp} and that $B_i^+(z)=B_i^-(z)^{-1}$. Substituting these identities back into \eqref{x+x-'}, we find that 
\begin{equation*}
 [\mbX_i^+(z),\mbX_i^-(w)]=\tfrac{1}{\hbar}\left(\delta({w+\hbar},{z})\mbH_i^+(w+\tfrac{\hbar}{2})- \delta({w},{z})\mbH_i^-(z)\right),
\end{equation*}
as desired. \qedhere
\end{proof}
%
%
\begin{proof}[The relation \eqref{DY:serre}]
Observe first that if $(\al_i,\al_j)=0$ then $\veps(\al_i,\al_j)=\veps(\al_j,\al_i)$ and \eqref{NO:2} implies 
\begin{equation*}
 [\mbX_i^\pm(z),\mbX_j^\pm(w)]=0.
\end{equation*}
Hence we only need to verify that \eqref{DY:serre} holds when $(\al_i,\al_j)=-1$. By \eqref{NO:2}, we have 
\begin{equation}\label{serre:1}
[\mbX_i^\pm(z_2),\mbX_j^\pm(w)]=\veps(\al_i,\al_j)\left(\frac{z_2^{-1}}{1-z_2^{-1}w_\mp}+\frac{w^{-1}}{1-w^{-1}(z_2)_\mp} \right) :\!\mbX_i^\pm(z_2)\mbX_j^\pm(w)\!:,
\end{equation}
while repeated application of \eqref{BA:1} gives 
\begin{align*}
 \mbX_i^\pm(z_1) :\! \mbX_i^\pm (z_2)\mbX_j^\pm(w)\!:&=-\veps(\al_i,\al_j)\frac{z_1(1-z_1^{-1}z_2)(1-z_1^{-1}(z_2\mp \hbar))}{1-z_1^{-1}w_\mp} :\! \mbX_i^\pm(z_1) \mbX_i^\pm (z_2)\mbX_j^\pm(w)\!:,\\
 :\! \mbX_i^\pm (z_2)\mbX_j^\pm(w)\!: \mbX_i^\pm(z_1)&=\veps(\al_i,\al_j)\frac{z_2^{2}w^{-1}(1-z_2^{-1}z_1)(1-z_2^{-1}(z_1\mp \hbar))}{1-w^{-1}(z_1)_\mp}:\! \mbX_i^\pm(z_1) \mbX_i^\pm (z_2)\mbX_j^\pm(w)\!:.
\end{align*}
Combining these last two identities with  \eqref{serre:1} gives 
\begin{equation*}
 \left[\mbX_i^\pm(z_1),[\mbX_i^\pm(z_2),\mbX_j^\pm(w)]\right]=-f(z_1,z_2,w):\! \mbX_i^\pm(z_1) \mbX_i^\pm (z_2)\mbX_j^\pm(w)\!:,
\end{equation*}
where 
\begin{equation*}
 f(z_1,z_2,w)=\left(\frac{z_2^{-1}}{1-z_2^{-1}w_\mp}+\frac{w^{-1}}{1-w^{-1}(z_2)_\mp} \right)\!\left(\frac{w^{-1}(z_2-z_1)(z_2-z_1\pm\hbar)}{1-w^{-1}(z_1)_\mp}                     
                +\frac{z_1^{-1}(z_1-z_2)(z_1-z_2\pm \hbar)}{1-z_1^{-1}w_\mp}\right).
\end{equation*}
Thus, the identity 
\begin{equation}
f(z_1,z_2,w)+f(z_2,z_1,w)=0 \label{serre:2}
\end{equation}
will imply $\left[\mbX_i^\pm(z_{1}),[\mbX_i^\pm(z_{2}),\mbX_j^\pm(w)]\right] + \left[\mbX_i^\pm(z_{2}),[\mbX_i^\pm(z_{2}),\mbX_j^\pm(w)]\right]=0$.  Since 
\begin{align*}
 \frac{z_2^{-1}}{1-z_2^{-1}w_\mp}+\frac{w^{-1}}{1-w^{-1}(z_2)_\mp} &=\delta(z_2,w_\mp)\mp \frac{\hbar w^{-2}}{(1-w^{-1}(z_2)_\mp)(1-w^{-1}(z_2)_\pm)}\quad \text{ and }\; \\
 \frac{z_1^{-1}}{1-z_1^{-1}w_\mp}&=\delta(z_1,w_\mp)-\frac{w^{-1}}{1-w^{-1}(z_1)_\pm},
\end{align*}
the property \eqref{delta} of the formal delta function implies that
\begin{align*}
 f(z_1,z_2,w)=  &\delta(z_2,w_\mp)\left(\frac{w^{-1}(z_2-z_1)(z_2-z_1\pm\hbar)}{1-w^{-1}(z_1)_\mp}                     
                +\frac{z_1^{-1}(z_1-z_2)(z_1-z_2\pm \hbar)}{1-z_1^{-1}w_\mp}\right)\\
                &\mp\delta(z_1,w_\mp)\frac{\hbar w^{-2}(z_1-z_2)(z_1-z_2\pm \hbar)}{(1-w^{-1}(z_2)_\mp)(1-w^{-1}(z_2)_\pm)}\\
                &\mp \frac{\hbar w^{-2}}{(1-w^{-1}(z_2)_\mp)(1-w^{-1}(z_2)_\pm)}\left(\frac{w^{-1}(z_2-z_1)(z_2-z_1\pm\hbar)}{1-w^{-1}(z_1)_\mp}                     
                -\frac{w^{-1}(z_1-z_2)(z_1-z_2\pm \hbar)}{1-w^{-1}(z_1)_\pm}\right)\\
                =&\pm\hbar  \delta(z_2,w_\mp)\mp  \hbar\delta(z_1,w_\mp)+\frac{\hbar^2w^{-4}(z_2-z_1)(z_2+z_1-2w)}{(1-w^{-1}(z_2)_\mp)(1-w^{-1}(z_2)_\pm)(1-w^{-1}(z_1)_\mp)(1-w^{-1}(z_1)_\pm)}.
\end{align*}
As this expression is antisymmetric in $z_1$ and $z_2$, we may conclude that \eqref{serre:2} holds, and thus that the vertex operators $\{\mbX_i^\pm(z)\}_{i\in I}$ satisfy the Serre relations \eqref{DY:serre}. 
\end{proof}
\let\qed\relax
\end{proof}
\begin{remark}\label{Rem:gen}
 Taking the coefficient of $z^{-2}w^0$ in the relation \eqref{DY:xxh} with $i=j$ yields 
 \begin{equation*}
  [x_{i1}^+,x_{i,-1}^-]=\msc+h_{i0}. 
 \end{equation*}
Combining this with the relation $[x_{i0}^+,x_i^-(w)]=\wt h_i^+(w+\frac{\hbar \msc}{2})+\wt h_i^-(w)$ (see \eqref{eqn:xi0xj-}), we deduce that 
$DY_\hbar^\msc(\mfg)$ is generated by $\{x_{ir}^\pm\}_{i\in I,r\in \Z}$. Moreover, in the Yangian double $DY_\hbar^\kappa(\mfg)$ at level $\kappa\in \C^\times$, the series $h_i^\pm(z)$ are uniquely determined by the relations
\begin{gather}
 \kappa^{-1}(z-w)[x_i^+(z),x_i^-(w)]=\delta({w+\hbar\kappa},{z})h_i^+\!\left(w+\tfrac{\hbar \kappa}{2}\right), \label{H+:un}\\
 \kappa^{-1}(z-w-\hbar\kappa)[x_i^+(z),x_i^-(w)]=\delta({w},{z})h_i^-(z). \label{H-:un}
\end{gather}
In particular, the representation $\rho_\hbar$ of Theorem \ref{thm:Vrep} is entirely determined by $x_i^\pm(z)\mapsto \mbX_i^\pm(z)$ for all $i\in I$, and the formulas \eqref{Hp:exp} and \eqref{Hm:exp} for $\mbH_i^\pm(z)$ may be deduced from \eqref{H+:un} and \eqref{H-:un}, as was essentially done below \eqref{x+x-'}. 
\end{remark}

\subsection{The $DY^\msc(\mfg)$-module $\widetilde\mcV$}\label{ssec:wtV}

As the coefficients of the vertex operators $\mbX_i^\pm(z)$ and $\mbH_i^\pm(z)$ are elements of $\End_{\C[\![\hbar]\!]}\mcV[\![\hbar]\!]$, it is not clear that they can be specialized at $\hbar=\zeta\in \C^\times$ to produce a $DY_\zeta^\msc(\mfg)$ representation.
In this subsection we exploit the existence of a $(\Z\times Q)$-grading on $\mcV$ to show that this can indeed be accomplished after modifying the representation space appropriately. 

The $(\Z\times Q)$-grading on $\mcV=\C[\mcH_{i,-r}]_{i\in I,r>0}\otimes \C_\veps[Q]$ is given by
\begin{equation*}
 \deg \mcH_{i,-r}=(-r,0), \quad \deg e^\al=\left(-\tfrac{1}{2}(\al,\al),\al \right) \quad \forall \; i\in I,\,r>0\; \text{ and }\;\al\in Q.  
\end{equation*}
Note that this choice of grading is different from the more familiar grading on Fock spaces obtained by setting $\deg \mcH_{i,-r}=(r,0)$ and $\deg e^\al=\left(\tfrac{1}{2}(\al,\al),\al \right)$. Let $\mcV_{n,\beta}$ denote the subspace of $\mcV$ spanned by elements of degree $(n,\beta)$, so that $\mcV=\bigoplus_{(n,\beta)\in \Z\times Q}\mcV_{n,\beta}$. We note the following useful observation:
\begin{lemma}\label{L:bounded}
 Setting $\mcP=\{(n,\beta)\in \Z\times Q\,:\,n\leq -\tfrac{1}{2}(\beta,\beta)\}$, we have 
 \begin{equation*}
  \mcV=\bigoplus_{(n,\beta)\in \mcP}\mcV_{n,\beta}.
 \end{equation*}
 Equivalently, $\mcV_{n,\beta}=\{0\}$ for all $n> -\tfrac{1}{2}(\beta,\beta)$. 
\end{lemma}
Next, for each $\beta \in Q$ we set $n(\beta)=-\frac{1}{2}(\beta,\beta)$, so that
\begin{equation*}
 \mcV_{\beta}=\bigoplus_{n\in \Z}\mcV_{n,\beta}=\bigoplus_{n\leq n(\beta)}\mcV_{n,\beta} \quad \forall \; \beta\in Q. 
\end{equation*}
Let $\wt \mcV_\beta=\prod_{n\leq n(\beta)}\mcV_{n,\beta}$ be the completion of $\mcV_\beta$ with respect to this grading, and set 
\begin{equation*}
 \wt \mcV=\bigoplus_{\beta \in Q}\wt \mcV_\beta. 
\end{equation*}
As $\mcV_0=\bigoplus_{n\leq 0}\mcV_{n,0}$ is precisely the Fock space $\mcF=\C[\mcH_{i,-r}]_{i\in I,r>0}$, we have the equivalent characterizations $\wt \mcV_\beta\cong\wt \mcF\ot \C e^\beta$ and $\wt \mcV\cong \wt \mcF\ot \C_\veps[Q]$, where $\wt \mcF=\wt \mcV_0$.

Now set 
\begin{equation*}
 \mcF_\hbar=(\C[\hbar])[H_{i,-r}]_{i\in I,r>0}\cong \C[\hbar]\otimes \mcF.
\end{equation*}
The $(\Z\times Q)$-grading on $\mcV$ extends to a grading on $\mcV_\hbar=\mcF_\hbar\otimes \C_\veps[Q]$ after imposing $\deg \hbar=(0,0)$. We use the same notation as above to denote its graded pieces and $\Z$-completion:
\begin{gather*}
 \mcV_\hbar=\bigoplus_{(n,\beta)\in \mcP}(\mcV_\hbar)_{n,\beta}=\bigoplus_{\beta\in Q}(\mcV_\hbar)_\beta \quad \text{ with }\quad (\mcV_\hbar)_\beta=\bigoplus_{n\leq n(\beta)}(\mcV_\hbar)_{n,\beta},\\
 \wt \mcV_\hbar=\bigoplus_{\beta\in Q}\wt{(\mcV_\hbar)}_\beta \cong \wt \mcF_\hbar\otimes \C_\veps[Q],
\end{gather*}
where $\wt{(\mcV_\hbar)}_\beta=\prod_{n\leq n(\beta)}(\mcV_\hbar)_\beta$ and $\wt \mcF_\hbar=\widetilde{(\mcV_\hbar)}_0$.

Recall that $\mbX_i^\pm(z)=\sum_{k\in \Z}\mbX_i^\pm[k]z^{-k-1}$ are the vertex operators which determine the action of $DY_\hbar^\msc(\mfg)$ on $\mcV[\![\hbar]\!]$ (see \eqref{Xpm:exp}).
\begin{proposition}\label{P:Vexp}
For each $k\in \Z$ and $i\in I$, $\mbX_i^\pm[k]$ admits an expansion 
\begin{equation}\label{deg-exp}
\mbX_i^\pm[k]=\sum_{a\geq 0} \mbX_i^\pm[k,a] \hbar^a,
\end{equation}
with $\mbX_i^\pm[k,a]\in \End_\C\mcV$ of degree $(k-a,\pm \al_i)$. Consequently, $\mbX_i^\pm(z)\in (\End_{\C[\hbar]}\wt \mcV_\hbar)[\![z,z^{-1}]\!]$ and the assignment
\begin{equation}\label{wt:rho_h}
x_i^\pm(z)\mapsto \mbX_i^\pm(z)\quad \forall \; i\in I,
\end{equation}
also determines an algebra morphism $\wt \rho_\hbar:DY_\hbar^\msc(\mfg)\to \End_{\C[\hbar]} \wt \mcV_\hbar$.
\end{proposition}
\begin{proof}
 The first part of the proposition is proven directly by expanding $\mbX_i^\pm(z)$ as a formal series in $\hbar$. To see that 
 $\mbX_i^\pm[k]\in \End_{\C[\hbar]}\wt \mcV_\hbar$ for each $k\in \Z$, it suffices to prove that $\mbX_i^\pm[k]\wt \mcV_\beta\subset \wt \mcV_{\beta\pm \alpha_i}$ for each $\beta\in Q$. This is a straightforward consequence of Lemma \ref{L:bounded} and \eqref{deg-exp}.
 
 One may prove the analogous statements for $\mbH_i^\pm(z)$ in the same way, but as noted in Remark \ref{Rem:gen} the coefficients of $x_i^\pm(z)$ generate $DY_\hbar^\msc(\mfg)$ (with $\mbH_i^\pm(z)$ uniquely determined by \eqref{H+:un} and \eqref{H-:un}), and hence this is not necessary and we may conclude that \eqref{wt:rho_h} determines an algebra morphism $\wt \rho_\hbar:DY_\hbar^\msc(\mfg)\to \End_{\C[\hbar]} \wt \mcV_\hbar$. \qedhere
\end{proof}
Proposition \ref{P:Vexp} implies that $\mbX_i^\pm(z)$ can be evaluated at $\hbar=\zeta\in \C$ to produce a well-defined element 
\begin{equation*}
 \mbX_i^\pm(z,\zeta)=\mbX_i^\pm(z)_{\hbar\mapsto \zeta}\in (\End_\C\wt \mcV)[\![z,z^{-1}]\!].
\end{equation*}
We will write $\mbX_i^{\pm,\zeta}[k]$ for the evaluation of $\mbX_i^\pm[k]$ at $\hbar=\zeta$, so that $\mbX_i^\pm(z,\zeta)=\sum_{k\in \Z}\mbX_i^{\pm,\zeta}[k] z^{-k-1}$.

Let $\End_m\mcV$ denote the subspace of $\End_\C \mcV$ spanned by operators of degree $m\in \Z$. (Here we consider only the $\Z$-grading on $\mcV=\bigoplus_{n\in \Z}\mcV_n$ induced by its $(\Z\times Q)$-grading.) Consider the direct product $\prod_{m\in \Z} \End_m \mcV$. The subspace 
\begin{equation*}
 \wt \End_\C \mcV=\left\{\sum_{m\in \Z}A_m\,:\, A_m=0 \quad \forall \; m\gg 0\right\}\subset \prod_{m\in \Z} \End_m \mcV
\end{equation*}
is an algebra with multiplication that respects the grading. 
\begin{corollary}\label{C:spec}
 For each $\zeta\in \C$, $\wt \rho_\hbar$ induces a homomorphism of $\C$-algebras 
 \begin{equation*}
  \rho_\zeta:DY_\zeta^\msc(\mfg)\to \End_\C \wt \mcV,\quad x_i^\pm(z)\mapsto \mbX_i^\pm(z,\zeta) \quad \forall \;i\in I.
 \end{equation*}
 Moreover, for each $k\in \Z$ we have 
 \begin{equation}\label{C:spec-2}
  \mbX_i^{\pm,\zeta}[k]\in \left(\prod_{m\leq k} \End_m\mcV \right)\cap \End_\C \wt \mcV \subset \wt \End_\C\mcV,
 \end{equation}
 and hence $\rho_\zeta$ may be viewed as a morphism $\rho_\zeta:DY_\zeta^\msc(\mfg)\to \wt \End_\C \mcV$. 
\end{corollary}
Henceforth, we will adapt the viewpoint that $\rho_\zeta$ has codomain $\wt \End_\C \mcV$, and we will focus almost exclusively on the case where $\zeta=1$, in which case we shall write $\rho=\rho_1$.  

By composing $\rho$ with $\iota:Y(\mfg)\to DY^\msc(\mfg)$ from Proposition \ref{P:iota}, we obtain an algebra morphism 
\begin{equation}\label{varrho}
 \varrho=\rho\circ\iota: Y(\mfg)\to \wt \End_\C \mcV.
\end{equation}
The algebra $\wt \End_\C\mcV$ admits a $\Z$-filtration $\{\mbF_k(\mcV)\}_{k\in \Z}$ given by 
\begin{equation*}
 \mbF_k(\mcV)=\prod_{m\leq k}\End_m\mcV,
\end{equation*} 
and we have
\begin{equation*}
 \mathrm{gr}_\Z \wt \End_\C\mcV=\bigoplus_{m\in \Z} \mbF_m(\mcV)/\mbF_{m-1}(\mcV)\cong \bigoplus_{m\in \Z}\End_m \mcV\subset \End_\C \mcV.  
\end{equation*}
Set $\mathring{\mbX}_i^\pm(z)=\sum_{k\in \Z}\mbX_i^\pm[k,0]z^{-k-1}\in (\End_\C\mcV)[\![z^{\pm 1}]\!]$, where $\mbX_i^\pm[k,0]$ is as in \eqref{deg-exp}. Explicitly, 
\begin{equation}\label{circ-X}
 \mathring{\mbX}_i^\pm(z)=\pm \exp\!\left(\pm\sum_{r>0}\frac{\wt\mcH_{i,-r}}{r}z^r\right) \exp\!\left(\mp \sum_{r>0} \frac{\mcH_{ir}}{r}z^{-r}\right)e^{\pm \al_i}z^{\partial_{\pm \al_i}} 
\end{equation}
with $\wt \mcH_{i,-r}=\sum_{j\in I}(\alpha_i,\alpha_j) \mcH_{j,-r}$ for each $i\in I$ and $r>0$. 
Since $\iota$ is a filtered morphism, the relation \eqref{C:spec-2} of Corollary \ref{C:spec} together with the expansion \eqref{deg-exp} implies the following. 
\begin{corollary}\label{C:rho_0}
$\rho$ and $\varrho$ are $\Z$-filtered morphisms, and the composition of $\gr \rho: \mathrm{gr}_\Z DY^\msc(\mfg)\to \End_\C\mcV$ with the morphism $\phi_D:U(\mft)\to \mathrm{gr}_\Z DY^\msc(\mfg)$ of Proposition \ref{P:wPBW-DY} is the representation
\begin{equation}\label{rho_0}
 \rho_0:U(\mft)\to \End_\C\mcV, \quad X_i^\pm(z)\mapsto \mathring{\mbX}_i^\pm(z) \quad \forall \; i\in I. 
\end{equation}
\end{corollary}
\begin{remark}
 Here it is understood that the $\Z_{\geq 0}$-filtration $\{\mbF_k\}_{k\geq 0}$ on $Y(\mfg)$ is extended to a $\Z$-filtration by setting 
 $\mbF_k=\{0\}$ for all $k<0$. The representation of $\mft$ given by \eqref{rho_0} can be obtained directly from $\rho_\hbar$ (see \eqref{Vrep}) by specializing $\hbar\mapsto 0$, or from $\rho_\zeta$ (see Corollary \ref{C:spec}) by taking $\zeta=0$. However, the $\Z$-filtration on $\wt \End_\C\mcV$ will play a crucial role in Section \ref{sec:PBW}.  
\end{remark}

\subsection{The $\mft$-modules $\mcV$ and $\mcV_\mbA$} 

By Corollary \ref{C:rho_0}, $\mcV$ admits the structure of a $\mft$-module with action encoded by the vertex operators $\mathring \mbX_i^\pm(z)$ defined in \eqref{circ-X}. When the Cartan matrix $\mbA$ is not invertible, this representation differs from that obtained from the classical construction of vertex representations \cite{FrKa,MRY}. In this subsection we explain the relation between the two constructions. 

We begin by recalling the classical setting. By \eqref{t:HH}, the Lie subalgebra of $\mft$ generated by the coefficients of the series $\{H_i(z)\}_{i\in I}$ is a homomorphic image of the following Heisenberg algebra.
\begin{definition}\label{D:h_A}
The Heisenberg Lie algebra $\mfH_\mbA$ associated to the Cartan matrix $\mbA$ (equivalently, to the root lattice $Q$) is the Lie algebra over $\C$ with basis $\{H_{ir}\}_{i\in I,r\in \Z}\cup\{\mathbf{C}\}$ subject to the defining Lie bracket relations
\begin{equation*}
 [H_{ir},\mathbf{C}]=0\quad \text{ and }\quad [H_{ir},H_{j,-s}]=r(\al_i,\al_j)\delta_{rs}\mathbf{C} \quad \forall \; i,j\in I \; \text{ and }\; r,s\in \Z. 
\end{equation*}
\end{definition}
\begin{NB}
 Note that we have included the family of elements $\{H_{i0}\}_{i\in I}$, which are central.
\end{NB}

For each fixed $\lambda\in Q$, there is a natural action of $\mfH_\mbA$ on the polynomial algebra $\C[H_{i,-r}]_{i\in I,r>0}$ given by 
\begin{equation*}
 H_{j,-s}(f)=H_{j,-s}f,\quad \mathbf{C}(f)=f,\quad H_{j0}(f)=(\alpha_j,\lambda)f,\quad H_{js}(f)=\partial_{js}(f)
\end{equation*}
for all $f\in \C[H_{i,-r}]_{i\in I,r>0}$, $j\in I$ and $s>0$, where $\partial_{js}$ is the derivation defined uniquely by  
\begin{equation*}
 \partial_{js}(H_{i,-r})=s(\al_i,\al_j)\delta_{sr}\quad \forall\; s>0\; \text{ and }\; i\in I.
\end{equation*}
We denote $\C[H_{i,-r}]_{i\in I,r>0}$, equipped with this $\mfH_\mbA$-module structure, by $\mcF_\mbA^\lambda$.

\begin{NB} 
The $\mfH_\mbA$-module $\mathscr{F}_\lambda$ is irreducible if and only if the form $(\,,\,)$ is non-degenerate on $Q$. 
This occurs precisely when the Cartan matrix $\mbA$ is invertible.
\end{NB}

Now define the vector space $\mcV_\mbA$ by 
\begin{equation*}
 \mcV_\mbA=\C[H_{i,-r}]_{i\in I,r>0}\ot \C_\veps[Q].
\end{equation*}
After identifying $\C[H_{i,-r}]_{i\in I,r>0}\ot \C e^\lambda$ with $\mcF_\mbA^\lambda$, the space $\mcV_\mbA$ becomes an $\mfH_{\mbA}$-module isomorphic to $\bigoplus_{\lambda\in Q}\mcF_\mbA^\lambda$. To extend this to a $\mft$-module structure, define for each $\al=\sum_{i\in I}n_i\al_i\in Q$ operators $\{H_{\al,r}\}_{r\in \Z}$ on $\mcV_\mbA$ by 
$H_{\al,r}=\sum_{i\in I}n_iH_{i,r}$. 
We then set 
\begin{equation*}
 \Gamma_\al^\pm(z)=\exp\left( \mp \sum_{r>0}\frac{H_{\al,\pm r}}{r}z^{\mp r}\right) \quad \forall\; \alpha\in Q,\\
\end{equation*}
and introduce vertex operators $\mdH_{\al}(z),\mdX_\al(z)\in (\End_\C \mcV_\mbA)[\![z^{\pm 1}]\!]$ by
\begin{equation*}
 \mathds{H}_\al(z)=\sum_{r\in \Z}H_{\alpha,r}z^{-r-1} \quad \text{ and }\quad \mathds{X}_\al(z)=\Gamma_\al^-(z)\Gamma_\al^+(z)e^\al z^{\partial_{\al}}\quad \forall \al\in Q.
\end{equation*}
\begin{proposition}\label{P:VrepA}
Set $\mathds{X}_i^\pm(z)=\pm \mathds{X}_{\pm\al_i}(z)$ and $\mathds{H}_i(z)=\mathds{H}_{\al_i}(z)$  for all $i\in I$. Then the assignment 
 \begin{equation}\label{Vrep:A}
 X_i^\pm(z)\mapsto \mathds{X}_i^\pm(z),\; H_i(z)\mapsto \mathds{H}_i(z) \quad \forall \; i\in I,\quad \mathbf{C}\mapsto 1
 \end{equation}
extends to a homomorphism of algebras $\rho_\mbA:U(\mft)\mapsto \End_\C \mcV_\mbA$. 

\begin{NB} 
The resulting $\mft$-module $\mcV_\mbA$ is irreducible if and only if $\mbA$ is invertible. 
\end{NB}

\end{proposition}
\begin{proof}
 Although, to the best of our knowledge, the statement of the proposition has only been written down explicitly for $\mbA$ of finite and of affine type \cite{FrKa,MRY}, the argument used to prove the above proposition for $\mft$ associated to the Cartan matrix of an arbitrary simply laced Kac-Moody algebra is the same, and analogous to the proof Theorem \ref{thm:Vrep}. We refer the reader to \cite[Thm.\ 14.8]{Ka}, \cite[Prop.\ 4.3]{MRY} and \cite[\S 6.5]{LeLi} for complete details. The result may also be deduced from \cite[Thm.\ 3.1]{Ji2}.\qedhere
\end{proof}
\begin{remark}\label{R:arb}
 Suppose now that $\mbA$ is the Cartan matrix of an arbitrary symmetric Kac-Moody algebra (not constrained by the condition \eqref{A-cond}), and let $\mft_\mbA$ be the Lie algebra defined identically to $\mft$ (see Definition \ref{D:t}), but with \eqref{t:XX} replaced by 
 \begin{equation*}
  (z-w)^{-a_{ij}}[X_i^\pm(z),X_j^\pm(w)]=0 \quad \forall \; i,j\in I. 
 \end{equation*}
 Then the assignment \eqref{Vrep:A} determines an algebra homomorphism $U(\mft_\mbA)\to \End_\C\mcV_\mbA$. The added difficulty in proving this statement is verifying that \eqref{Vrep:A} preserves the Serre relation \eqref{t:Serre} when $a_{ij}<-1$. This can again be deduced from \cite{Ji2}, although it may also be proven directly using elementary properties of the formal delta function $\delta({z},{w})$ and its partial derivatives. 
 
 \begin{NB}
  The proof boils down to the following two facts: 
  \begin{enumerate}
   \item Let $f(z)\in R[\![z^\pm]\!]$ for some complex vector space $R$. Then 
 \begin{equation*}
  f(z)\delta_{z}^{(n)}({z},{w})=\sum_{k=0}^n (-1)^k \binom{n}{k} \frac{d^a}{dz^a}f(z)|_{z=w} \delta_z^{(n-a)}({z},{w})\quad \forall n\geq 0.
 \end{equation*}
 
  \item For each $n\geq 0$, $(z-w)^{n+1}\delta_{z}^{(n)}({z},{w})=0$. 
  \end{enumerate}
  There is a complete proof in our shared folder, but I guess we don't need it anymore. 
 \end{NB}
\end{remark}
We now turn to relating $\mcV_\mbA$ with the $\mft$-module $\mcV$ from Corollary \ref{C:rho_0}. 
Recall from Definition \ref{D:h} that $\mfH$ is the Heisenberg Lie algebra associated to the trivial lattice $\Z^{|I|}$. For each $k\in \Z_{\neq 0}$, set 
\begin{equation*}
 \mfH_\mbA^{(k)}=\bigoplus_{i\in I}\C H_{ik}\quad \text{ and }\quad \mfH^{(k)}=\bigoplus_{i\in I}\C \mcH_{ik}.
\end{equation*}
Similarly, we set 
$\mfH_\mbA^{(0)}=\bigoplus_{i\in I}\C H_{i0} \oplus \C\cdot \mathbf{C}$ and $\mfH^{(0)}=\C\cdot \mathcal{C}$. Let $\mfH_\mbA^\prime$ be the Lie subalgebra of $\mfH_\mbA$ defined by
\begin{equation*}
 \mfH_\mbA^\prime=\bigoplus_{k\neq 0} \mfH_\mbA^{(k)} \oplus \C \cdot \mathbf{C}=\bigoplus_{k\in \Z} \mfH_\mbA^{\prime (k)}, \quad \text{ where } \; \mfH_A^{\prime (k)}=\mfH_\mbA^{(k)}\cap \mfH_\mbA^\prime.
\end{equation*}
In addition, we denote $\bigoplus_{k\geq 0}\mfH_A^{\prime (k)}$ by $\mfH_\mbA^+$ and 
$\bigoplus_{k< 0}\mfH_A^{\prime (k)}$ by $\mfH_\mbA^-$, and define $\mfH^\pm$ analogously. 
The following lemma is straightforward.
\begin{lemma} \label{L:h_A->h}
\leavevmode
\begin{enumerate}
 \item The assignment 
 \begin{equation*}
 \varphi_\mbA: \mathbf{C}\mapsto \mcC, \quad H_{ir}\mapsto \wt H_{ir}=\begin{cases}
                                                                       \sum_{j\in I}(\alpha_i,\alpha_j)\mcH_{jr} \quad &\text{ if }\; r<0,\\
                                                                       \mcH_{ir} \quad &\text{ if }\; r>0,\\
                                                                      \end{cases}
 \end{equation*}
 extends to a morphism of graded Lie algebras $\varphi_\mbA:\mfH_\mbA^\prime \to \mfH$. 
 \item For each $r<0$, $\varphi_\mbA|_{\mfH^{(r)}_\mbA}:\mfH^{(r)}_\mbA\to \mfH^{(r)}$ has matrix equal to $\mbA$ with respect to the bases $\{H_{ir}\}_{i\in I}\subset \mfH_\mbA^{(r)}$ and $\{\mcH_{ir}\}_{i\in I}\subset \mfH_\mbA^{(r)}$. Consequently, 
 \begin{equation*}
 \varphi_\mbA|_{\mfH_\mbA^-}:\mfH_\mbA^-\to \mfH^-
 \end{equation*}
 is an isomorphism if and only if $\mbA$ is invertible, and the same is true for $\varphi_\mbA$. 
\end{enumerate}
\end{lemma}
By the lemma, $\varphi_\mbA|_{\mfH_\mbA^-}:\mfH_\mbA^-\to \mfH^-$ induces an algebra morphism 
\begin{equation*}
 \Phi_\mbA:U(\mfH_\mbA^-)\to U(\mfH^-)
\end{equation*}
which is invertible precisely when $\mbA$ is. After identifying $U(\mfH_\mbA^-)$ and $U(\mfH^-)$ with the Fock space representations 
$\C[H_{i,-r}]_{i\in I,r>0}$ and $\C[\mcH_{i,-r}]_{i\in I,r>0}$ of $\mfH_\mbA^\prime$ and $\mfH$, respectively, and equipping $\C[\mcH_{i,-r}]_{i\in I,r>0}$ with the structure of a $\mfH_\mbA^\prime$-module via $\varphi_\mbA$, $\Phi_\mbA$ becomes a morphism of $\mfH_\mbA^\prime$-modules. This discussion leads us to the following result. 
\begin{proposition}\label{P:V_A->V}
The $\C$-algebra morphism 
\begin{equation*}
 \Phi_\mbA\otimes \mathrm{id}: \C[H_{i,-r}]_{i\in I,r>0}\otimes \C_\veps[Q]\to \C[\mcH_{i,-r}]_{i\in I,r>0}\otimes \C_\veps[Q]
\end{equation*}
is a morphism of $\mft$-modules $\mcV_\mbA\to \mcV$. It is an isomorphism precisely when $\mbA$ is invertible. 
\end{proposition}
\begin{proof}
 Lemma \ref{L:h_A->h} and the discussion following it prove that $\Phi_\mbA \otimes \mathrm{id}$ will be invertible exactly when $\mbA$ is. By comparing the definitions of the vertex operators $\mathring{\mbX}_i^\pm(z)$ and $\mdX_i^\pm(z)$ (see \eqref{circ-X} and Proposition \ref{P:VrepA}), we find that $\Phi_\mbA\otimes \mathrm{id}$ will be a morphism of $\mft$-modules provided $\Phi_\mbA$ is a morphism of $\mfH_\mbA^\prime$-modules in the sense described before the statement of the proposition. As this has already been established, the proposition is proved. 
\end{proof}

\section{The Poincar\'{e}-Birkhoff-Witt Theorem}\label{sec:PBW}

We now fix $\mfg$ to be a Kac-Moody algebra associated to an indecomposable Cartan matrix $\mbA$ which is of affine type, and whose associated Dynkin diagram is simply laced with $\ell+1$ vertices. As in Section \ref{Sec:aff}, we set $I=\{0,1,\ldots, \ell\}$ with $\{1,\ldots,\ell\}$ labeling the Dynkin diagram of the underlying finite-dimensional rank $\ell$ simple Lie algebra $\mfg_0$.

In this section we will prove that the epimorphism $\phi:U(\mfs)\onto \gr Y(\mfg)$ of Proposition \ref{P:weakPBW} is an isomorphism: see Theorem \ref{T:PBW}. By Propositions \ref{P:tuce} and \ref{P:g'[t]}, this will imply that $\gr Y(\mfg)\cong U(\mathfrak{uce}(\mfg^\prime[t]))$. As a corollary, we prove in Theorem \ref{T:flat} that $Y_\hbar(\mfg)$ is a flat deformation of $U(\mfs)\cong U(\mathfrak{uce}(\mfg^\prime[t]))$ (see Remark \ref{R:flat}). 

\subsection{A faithful representation of \texorpdfstring{$\mfs$}{}}

Our first step in proving the injectivity of $\phi$ is to use the results of Section \ref{Sec:Ver} to produce a representation of $Y(\mfg)$ which specializes to a faithful representation of $\mfs\cong \mathfrak{uce}(\mfg^\prime[t])$. To accomplish this, we first enlarge $\mbA$ to an invertible Cartan matrix.

Set $\mathring{I}=I\cup\{-1\}$, and extend $\mbA$ to a Cartan matrix $\mathring{\mbA}= (a_{ij})_{i,j\in \mathring{I}}$ by imposing 
\begin{equation*}
 a_{-1,i}=a_{i,-1}=2\delta_{-1,i}-\delta_{i,0} \quad \forall \; i\in \mathring{I}.
\end{equation*}
\begin{definition}
Define $\mathring{\mfg}$ to be the simply-laced Kac-Moody algebra with  Cartan matrix $\mathring{\mbA}$.
\end{definition}
We fix an invariant symmetric non-degenerate bilinear form $\langle\,,\,\rangle$ on $\mathring{\mfg}$ extending $(\,,\,)$, and assume that it is normalized so that $\langle \alpha_i,\alpha_i\rangle=2$ for all $-1\leq i\leq \ell$. In particular $a_{ij}=\langle \alpha_i,\alpha_j\rangle$ for all $-1\leq i,j\leq \ell$. Let $\mathring{Q}=\bigoplus_{-1\leq i\leq \ell}\Z \alpha_i=\Z\alpha_{-1}\oplus Q$ denote the root lattice of $\mathring{\mfg}$. The following lemma can be easily deduced.
\begin{lemma}\label{L:invA}
 The Cartan matrix $\mathring{\mbA}$ is invertible. In particular, $\langle \,,\,\rangle|_{\mathring{Q}\times \mathring{Q}}$ is non-degenerate.  
\end{lemma}
Henceforth, we will use the notation $\mathring{\mcV}$ to denote the space \eqref{V} corresponding to the above data: 
\begin{equation*}
 \mathring{\mcV}=\C[\mcH_{i,-r}]_{i\in \mathring{I},r>0}\otimes \C_\veps[\mathring{Q}]. 
\end{equation*}
By Corollaries \ref{C:spec} and \ref{C:rho_0}, we have a $\Z$-filtered morphism of $\C$-algebras 
\begin{equation*}
 \mathring\rho:DY^\msc(\mathring{\mfg})\to \wt\End_\C \mathring{\mcV}, \quad x_i^\pm(z)\mapsto \mbX_i^\pm(z,1) \quad \forall \; i\in \mathring{I}. 
\end{equation*}
Observe that the assignment 
\begin{equation*}
\mathring{\iota}:x_i^\pm(z) \mapsto x_i^\pm(z), \; h_i^\pm(z) \mapsto h_i^\pm(z), \; \msc \mapsto \msc \quad \forall \; i\in I
\end{equation*}
extends to a filtered algebra homomorphism $\mathring{\iota} :DY^\msc(\mfg)\to DY^\msc(\mathring{\mfg})$. We set 
\begin{equation*}
 \mathbullet{\rho}=\mathring{\rho}\circ \mathring{\iota}: DY^\msc(\mfg)\to \wt\End_\C \mathring{\mcV}.
\end{equation*}
 Define a representation $\mathbullet{\rho}_0$ of $\mft$ on $\mathring{\mcV}$ by setting 
 \begin{equation}\label{vrho_0}
 \mathbullet{\rho}_0=\gr(\mathring{\rho}\circ\mathring{\iota})\circ \phi_D:U(\mft)\to \bigoplus_{m\in \Z}\End_m\mathring{\mcV}\subset \End_\C\mathring{\mcV},
 \end{equation}
 where $\phi_D$ is as in Proposition \ref{P:wPBW-DY}. 
\begin{lemma}
 The $\mft$-module $\mathring{\mcV}$, equipped with action given by $\mathbullet{\rho}_0$ above, is a faithful module. 
\end{lemma}
\begin{proof} 
 Let $\mathring{\mft}$ be the Lie algebra from Definition \ref{D:t} corresponding to $\mathring{\mbA}$. By \eqref{P:V_A->V} and Lemma \ref{L:invA}, the morphism 
  $\Phi_{\mathring{\mbA}}\otimes \mathrm{id}: \mcV_{\mathring{\mbA}}\to \mathring{\mcV}$ is an isomorphism of $\mathring{\mft}$-modules.  By pulling back via the natural morphism $\mft\to \mathring{\mft}$, we obtain an isomorphism of $\mft$-modules, and the induced $\mft$-module structure on $\mathring{\mcV}$ is precisely that given by $\mathbullet{\rho}_0$. That this is a faithful $\mft$-module now follows from the fact that $\mcV_{\mathring{\mbA}}$ is precisely $V(\mathring{Q},\mfH_{\mathring{\mbA}}^\prime)$ from \cite{MRY}, and by \cite[Prop.\ 4.3]{MRY}, this is a faithful $\mft$-module. 
\end{proof}
Now set $\mathbullet{\varrho}=\mathbullet{\rho}\circ \iota:Y(\mfg)\to \wt \End_\C \mathring{\mcV}$, where $\iota:Y(\mfg)\to DY^\msc(\mfg)$ is as in Proposition \ref{P:iota}, and define 
\begin{equation}
 \mathbullet{\varrho}_0=\gr\mathbullet{\varrho}\circ\phi: U(\mfs)\to \bigoplus_{m\in \Z}\End_m\mathring{\mcV}\subset \End_\C \mathring{\mcV}, \label{bvrho0}
\end{equation}
where $\phi:U(\mfs)\to \gr Y(\mfg)$ is as in Proposition \ref{P:weakPBW}.
\begin{corollary}\label{C:faith-s}
 The $\mfs$-module $\mathring{\mcV}$, equipped with action given by $\mathbullet{\varrho}_0$ above, is a faithful module.
\end{corollary}
\begin{proof}
 The representation $\mathbullet{\varrho}_0$ is equal to the restriction of $\mathbullet{\rho}_0$ to $U(\mfs)$ via the embedding of Corollary \ref{C:s->t}, so the result follows immediately.  
\end{proof}

We will use this faithful module, together with the coproduct $\Delta_{\mfs,u}$ from Subsection \ref{ssec:Delta_u}, to construct an embedding of $U(\mfs)$ into a large algebra built by gluing together endomorphism rings associated to $\mathring{\mcV}$. We begin with the following general result.

Let $\mfa$ be an arbitrary complex Lie algebra and 
let $\Delta_\mfa$ and $\veps_\mfa$ be the coproduct and counit, respectively, of the enveloping algebra $U(\mfa)$. 
\begin{theorem}\label{T:Phi}
Let $V$ be a faithful representation of $\mfa$ with $\pi:U(\mfa)\to \End_\C V$ the corresponding homomorphism. For each 
$k\geq 0$, set $\pi_{k}=\rho_\mfa^{\otimes k}\circ\Delta_\mfa^{(k-1)}$, with $\pi_{0}=\veps_\mfa$. The universal property of $\prod_{m\geq 0}\End_\C(V^{\ot m})$ dictates that there is a unique morphism
\begin{equation*}
 \Phi:U(\mfa)\to \prod_{m\geq 0}\End_\C(V^{\ot m}),\quad \mathrm{pr}_m\circ \Phi=\pi_{m}\quad \forall\; m\geq 0,  
\end{equation*}
where $\mathrm{pr}_m:\prod_{m\geq 0}\End(V^{\ot m})\to \End(V^{\ot m})$ is the natural projection. Then $\Phi$ is injective. 
\end{theorem}
\begin{proof}
The proof can be found in Appendix \ref{app:A}: see Theorem \ref{T:PBW2}.
\end{proof}
(See Lemma 3.5 and its proof in \cite{AMR} for a similar result.)
Now, we would like to imitate Theorem \ref{T:Phi} with $\Delta_\mfa^{(k-1)}$ replaced by $\Delta_{\mfs,u}^{k-1}$. 
Let $V$ be a faithful $\mfs$-module with corresponding homomorphism $\rho_\mfs:U(\mfs)\to \End_\C V$, and for each $k\geq 1$, set 
\begin{equation*}
 \rho_{\mfs,u}^k=\rho_\mfs^{\otimes k}\circ \Delta_{\mfs,u}^{k-1}:U(\mfs)\to \End_\C(V^{\ot k})[u^{\pm 1}].
\end{equation*}
We also set $\rho_{\mfs,u}^0=\veps_\mfs:U(\mfs)\to \C\subset \C[u^{\pm 1}]$. Then there is a unique morphism
\begin{equation}\label{Phi_u}
 \Phi_u:U(\mfs)\to \prod_{m\geq 0}\End_\C(V^{\ot m})[u^{\pm 1}],\quad \mathrm{pr}_m\circ \Phi_u=\rho_{\mfs,u}^m\quad \forall\; m\geq 0,  
\end{equation}
where $\mathrm{pr}_m:\prod_{m\geq 0}\End_\C(V^{\ot m})[u^{\pm 1}]\to \End_\C(V^{\ot m})[u^{\pm 1}]$ is the natural projection.
\begin{proposition}\label{P:Phi_u}
The morphism $\Phi_u$ is injective.
\end{proposition}
\begin{proof}
 The evaluation $u\mapsto 1$ induces a morphism 
\begin{equation*}
  \mathrm{ev}: \prod_{m\geq 0}\End_\C(V^{\ot m})[u^{\pm 1}]\to \prod_{m\geq 0}\End_\C(V^{\ot m}). 
\end{equation*}
The composition $\mathrm{ev}\circ\Phi_u:U(\mfs)\to \prod_{m\geq 0}\End_\C(V^{\ot m})$ agrees with the morphism $\Phi$ associated to $V$ 
from Theorem \ref{T:Phi}, and hence is injective. This implies that $\Phi_u$ is also injective. 
\end{proof}

Applying Proposition \ref{P:Phi_u} with $V$ the faithful $\mfs$-module $\mathring{\mcV}$ from Corollary \ref{C:faith-s}, we obtain the following corollary.
\begin{corollary}\label{C:faith-U(s)}
 The morphism of $\C$-algebras $\Phi_u:U(\mfs)\to \prod_{m\geq 0}\End_\C(\mathring{\mcV}^{\ot m})[u^{\pm 1}]$, defined by \eqref{Phi_u} with $\rho_\mfs=\mathbullet{\varrho}_0$, is injective. 
\end{corollary}
%
%

\subsection{Statement and proof of the main result}\label{Ss:proofmain}
We now construct the Yangian version $\Psi_u$ of the embedding $\Phi_u$ from Corollary \ref{C:faith-U(s)}, using the morphism $\mathbullet{\varrho}:Y(\mfg)\to \wt \End_\C \mathring{\mcV}$. The injectivity of $\Psi_u$ is closely tied to the Poincar\'{e}-Birkhoff-Witt Theorem for $Y(\mfg)$, as we shall explain shortly. 

For each $k\geq 1$, $\mathbullet{\varrho}^{\otimes k}$ extends to a homomorphism 
$Y(\mfg)^{\ot k}(\!(u)\!)\to (\wt \End_\C \mathring{\mcV})^{\ot k}(\!(u)\!)$.
Composing with $\Delta_u^{k-1}$ from \eqref{Delta_u^k:2}, we obtain a morphism  
\begin{equation*}
 \mathbullet{\varrho}_u^k:Y(\mfg)\to (\wt \End_\C\mathring{\mcV})^{\ot k}(\!(u)\!).
\end{equation*}
As in the $U(\mfs)$-case, we set $\mathbullet{\varrho}_u^0$ to be the counit. 

For each $a\in \Z$, set 
\begin{equation*}
\End_a(\mathring{\mcV}^{\ot k})=\bigoplus_{a_1+\ldots+a_{k}=a}\left(\End_{a_1}\mathring{\mcV}\otimes \cdots \otimes \End_{a_{k}}\mathring{\mcV}\right)\subset \End_\C(\mathring{\mcV}^{\ot k}).
\end{equation*}
We let $\wt \End_\C(\mathring{\mcV}^{\ot k})$ denote the subspace of $\prod_{a\in \Z}\End_a(\mathring{\mcV}^{\ot k})$ consisting of summations $\sum_{a\in \Z}A_a$ with $A_a=0$ for all $a\gg 0$. This is an algebra with multiplication extending that of $\bigoplus_{a\in \Z} \End_a(\mathring{\mcV}^{\ot k})$. Setting
\begin{equation*}
 \mbF_\ell\!\left(\wt \End_\C(\mathring{\mcV}^{\ot k}) \right) =\prod_{a\leq \ell} \End_a(\mathring{\mcV}^{\ot k}) \quad \forall\; \ell\in \Z 
\end{equation*}
equips $\wt \End_\C(\mathring{\mcV}^{\ot k})$ with the structure of a $\Z$-filtered algebra. Recall that $\{\mbF_\ell\}_{\ell\geq 0}$ denotes the $\Z_{\geq 0}$-filtration on $Y(\mfg)$ defined above Proposition \ref{P:weakPBW}, which is extended to a $\Z$-filtration by setting $\mbF_{-\ell}=0$ for $\ell>0$. 
\begin{lemma}\label{L:fil-k}
 The image of $\mathbullet{\varrho}_u^k$ embeds into $\wt \End_\C(\mathring{\mcV}^{\ot k})(\!(u)\!)$. Moreover, 
 \begin{equation*}
  \mathbullet{\varrho}_u^k(\mbF_\ell)\subset \mbF_\ell\!\left(\wt \End_\C(\mathring{\mcV}^{\ot k}) \right)(\!(u)\!) \quad \forall \; \ell\in \Z. 
 \end{equation*}
\end{lemma}
\begin{proof}
 By \eqref{Delta_u^k:3}, $\Delta_u^{k-1}(\mbF_\ell)\subset \mbF_\ell(Y(\mfg)^{\otimes k})(\!(u)\!)$, where 
  $\mbF_\ell(Y(\mfg)^{\otimes k})=\sum_{a_1+\ldots+a_k=\ell}\mbF_{a_1}\otimes \cdots\otimes \mbF_{a_k}$. Since $\mathbullet \varrho$ is also filtered (see Corollary \ref{C:rho_0}), we have $\mathbullet{\varrho}_u^k(\mbF_\ell)\subset \wt\mbF_\ell\!\left(\wt \End_\C(\mathring{\mcV}^{\ot k}) \right)$, with  
\begin{equation*}
   \wt\mbF_\ell\!\left(\wt \End_\C(\mathring{\mcV}^{\ot k}) \right)= \sum_{a_1+\ldots+a_k=\ell}\mbF_{a_1}\!\left(\wt \End_\C(\mathring{\mcV}) \right)\otimes\cdots\otimes\mbF_{a_k}\!\left(\wt \End_\C(\mathring{\mcV}) \right).
\end{equation*}
As $\mbF_m\!\left(\wt \End_\C(\mathring{\mcV}) \right)=\prod_{a\leq m} \End_a(\mathring{\mcV})$, we see that $\bigotimes_{b=1}^k\mbF_{a_b}\!\left(\wt \End_\C(\mathring{\mcV}) \right)$ naturally embeds into the space $\mbF_\ell\!\left(\wt \End_\C(\mathring{\mcV}^{\ot k}) \right)=\prod_{a\leq \ell} \End_a(\mathring{\mcV}^{\ot k})$, provided $\sum_{b=1}^k a_b=\ell$. This proves the assertion. 
\end{proof}
Consider the algebra
\begin{equation*}
 \End_u(\mathring{\mcV}^{\ot k})=\bigcup_{\ell\in \Z}\left(\mbF_\ell\!\left(\wt \End_\C(\mathring{\mcV}^{\ot k}) \right)(\!(u)\!)\right)\subset \End_\C(\mathring{\mcV}^{\ot k})(\!(u)\!).
\end{equation*}
It is $\Z$-filtered with $\mbF_\ell(\End_u(\mathring{\mcV}^{\ot k}))=\mbF_\ell\!\left(\wt \End_\C(\mathring{\mcV}^{\ot k}) \right)(\!(u)\!)$ and 
\begin{equation*}
 \gr_{\mathbb{Z}}\End_u(\mathring{\mcV}^{\ot k})=\bigoplus_{\ell\in \Z}\End_{\ell}(\mathring{\mcV}^{\otimes k})(\!(u)\!)\subset \End_\C(\mathring{\mcV}^{\ot k})(\!(u)\!).
\end{equation*}
Lemma \ref{L:fil-k} implies that $\mathbullet{\varrho}_u^k$ can be viewed as a $\Z$-filtered morphism 
\begin{equation*}
 \mathbullet{\varrho}_u^k:Y(\mfg)\to \End_u(\mathring{\mcV}^{\ot k}).
\end{equation*}
After forming the direct product of algebras $\prod_{m\geq 0} \End_u(\mathring{\mcV}^{\ot m})$, we obtain an algebra morphism
\begin{equation}
 \Psi_u:Y(\mfg)\to \prod_{m\geq 0} \End_u(\mathring{\mcV}^{\ot m}), \quad  \mathrm{pr}_m\circ \Psi_u=\mathbullet{\varrho}_u^m \quad \forall \; m\geq 0, \label{Psi}
\end{equation}
where $\mathrm{pr}_m:\prod_{m\geq 0} \End_u(\mathring{\mcV}^{\ot m})\to  \End_u(\mathring{\mcV}^{\ot m})$ is the $m$-th projection morphism. 
We are now ready to state and prove the main result of this section: 
\begin{theorem}\label{T:PBW}
The morphism $\Psi_u$ defined in \eqref{Psi} is an embedding of algebras, and the epimorphism 
 \begin{equation*}
  \phi:U(\mfs)\cong U(\mathfrak{uce}(\mfg^\prime[t]))\onto \gr Y(\mfg), \quad X_{ir}^\pm \mapsto \bar{x}_{ir}^\pm, \; H_{ir}\mapsto \bar{h}_{ir}
 \end{equation*}
 of Proposition \ref{P:weakPBW} is an isomorphism of algebras. 
\end{theorem}
\begin{proof}
 As, for each $k\geq 0$, $\mathbullet{\varrho}_u^k$ is a filtered morphism $Y(\mfg)\to \End_u(\mathring{\mcV}^{\ot k})$, we may form the associated graded morphisms 
 \begin{equation*}
  \gr \mathbullet{\varrho}_u^k:\gr Y(\mfg)\to  \End_\C(\mathring{\mcV}^{\ot k})(\!(u)\!). 
 \end{equation*}
 By \eqref{grDelta_u^k}, the image of $\gr \mathbullet{\varrho}_u^k$ in fact lies in $\End_\C(\mathring{\mcV}^{\ot k})[u^{\pm 1}]$. We therefore obtain an algebra morphism
 \begin{equation*}
 \overline{\Psi}_u:\gr Y(\mfg)\to \prod_{m\geq 0}\End_\C(\mathring{\mcV}^{\ot m})[u^{\pm 1}], \quad \mathbullet{\mathrm{pr}}_m\circ \overline{\Psi}_u=\gr \mathbullet{\varrho}_u^m \quad \forall \;m\geq 0,
 \end{equation*}
where now $\mathbullet{\mathrm{pr}}_m$ is the $m$-th projection morphism for $\prod_{m\geq 0}\End_\C(\mathring{\mcV}^{\ot m})[u^{\pm 1}]$. 

By definition,  $\mathbullet{\varrho}_0=\gr\mathbullet{\varrho} \circ \phi$ (see \eqref{bvrho0}), and hence the commutativity of the diagram \eqref{Delta_u:com} implies that 
\begin{equation} \label{Psi_u}
 \overline{\Psi}_u \circ \phi = \Phi_u,
\end{equation}
where $\Phi_u:U(\mfs)\to \prod_{m\geq 0}\End_\C(\mathring{\mcV}^{\ot m})[u^{\pm 1}]$ is the embedding of Corollary \ref{C:faith-U(s)}. This implies that $\phi$ is also injective, and hence an isomorphism of algebras. 

The relation \eqref{Psi_u}, together with the just proven fact that $\phi$ is an isomorphism, also implies that $ \overline{\Psi}_u$  is an embedding, from which it follows that $\Psi_u$ is injective using a standard argument. Indeed, given a nonzero element $X\in Y(\mfg)$, 
we may take $\ell\geq 0$ minimal such that $X\in \mbF_\ell$. Let $\bar{X}\in  \gr Y(\mfg)$ denote the image of $X$ in $\mbF_{\ell}/\mbF_{\ell-1}$, which is nonzero by assumption. If $\Psi_u(X)=0$, then $\gr \mathbullet{\varrho}_u^m(\bar X)=0$ for all $m\geq 0$ and hence $\overline{\Psi}_u(\bar X)=0$, which is impossible.   \qedhere
\end{proof}
Recall that if $\mbA$ is a $\Z_{\geq 0}$-filtered algebra with ascending filtration $\{\mbF_k(\mbA)\}_{k\geq 0}$, then the Rees algebra associated to $\mbA$ is 
\begin{equation*}
 R_\hbar(\mbA)=\bigoplus_{k\geq 0}\hbar^k \mbF_k(\mbA)\subset \mbA[\hbar]. 
\end{equation*}
The Rees algebra $R_\hbar(\mbA)$ always satisfies $R_\hbar(\mbA)/(\hbar-1)R_\hbar(\mbA)\cong \mbA$ and $R_\hbar(\mbA)/\hbar R_\hbar(\mbA)\cong \gr \mbA$. The next theorem employs the Rees algebra construction to characterize $Y_\hbar(\mfg)$ in terms of $Y(\mfg)$. 
\begin{theorem}\label{T:flat}
 The assignment $x_{ir}^\pm\mapsto \hbar^r x_{ir}^\pm$, $h_{ir}\mapsto \hbar^r h_{ir}$ extends to an isomorphism of $\C[\hbar]$-algebras 
 \begin{equation*}
  \Psi_\hbar:Y_\hbar(\mfg)\to R_\hbar(Y(\mfg))\subset Y(\mfg)[\hbar].
 \end{equation*}
 Consequently, $Y_\hbar(\mfg)$ is a flat deformation of the algebra $U(\mfs)\cong U(\mathfrak{uce}(\mfg^\prime[t]))$.
\end{theorem}
\begin{proof}
 That the assignment $x_{ir}^\pm,h_{ir}\mapsto \hbar^r x_{ir}^\pm, \hbar^r h_{ir}$ extends to a homomorphism $\Psi_\hbar$ of $\C[\hbar]$-algebras is verified directly (cf. \eqref{zeta-iso}). Since $\{\hbar^r x_{ir}^\pm, \hbar^r h_{ir}\}_{i\in I,r\geq 0}$ generate $R_\hbar(Y(\mfg))$ as a $\C[\hbar]$-algebra, $\Psi_\hbar$ is surjective. 
 
 The composition $\varpi$ of the isomorphism $R_\hbar(Y(\mfg))/\hbar R_\hbar(Y(\mfg))\to \gr Y(\mfg)$ with the quotient homomorphism $R_\hbar(Y(\mfg))\to R_\hbar(Y(\mfg))/\hbar R_\hbar(Y(\mfg))$ satisfies 
 \begin{equation*}
  \hbar^k x_{ik}^\pm \mapsto\bar{x}_{ik}^\pm,\; \hbar^k h_{ir} \mapsto \bar{h}_{ik} \quad \forall \; i\in I\; \text{ and }\; k\geq 0. 
 \end{equation*}
 Moreover, $\varpi\circ \Psi_\hbar$ sends the ideal $\hbar Y_\hbar(\mfg)$ to zero and thus factors through the quotient $Y_\hbar(\mfg)/\hbar Y_\hbar(\mfg)$ to give 
 $\Psi_0:Y_0(\mfg)\to \gr Y(\mfg)$. After using Proposition \ref{P:s->Yh} and Theorem \ref{T:PBW} to identify both the domain and codomain of $\Psi_0$ with $U(\mfs)$, $\Psi_0$ becomes the identity map. 
 
 Now suppose that there is a nonzero $X\in \mathrm{Ker} \Psi_\hbar$. Let $m$ be the maximal non-negative integer such that $X=\hbar^m Y$ for some  $Y\in Y_\hbar(\mfg)$ (that $m$ is finite follows from the fact that $Y_\hbar(\mfg)$ is $\Z_{\ge 0}$-graded with $\deg \hbar=1$). Since $\Psi_\hbar$ is a $\C[\hbar]$-algebra morphism and $R_\hbar(Y(\mfg))$ is torsion free, $Y\in \mathrm{Ker}\Psi_\hbar$. By maximality of $m$, the image $\bar Y$ of $Y$ in $Y_0(\mfg)$ is nonzero. Since $\Psi_0:Y_0(\mfg)\to \gr Y(\mfg)$ is an isomorphism, $\Psi_0(\bar Y)\neq 0$. This is a contradiction as $\Psi_0(\bar Y)=\overline{\Psi_\hbar(Y)}=0$. Therefore $\Psi_\hbar$ is injective, and thus an isomorphism.
 
To prove that $Y_\hbar(\mfg)$ is a flat deformation of $U(\mfs)$, it remains to see that it is flat (or equivalently, torsion free) as a $\C[\hbar]$-module. This is a consequence of the fact that it embeds into the torsion free space $Y(\mfg)[\hbar]$, and hence is itself torsion free. 
\end{proof}

\appendix
\section{} \label{app:A}

Let $\pi:U(\mfg)\to \End(V)$ be a faithful representation of an arbitrary complex Lie algebra $\mfg$. Let $\Delta$ and $\veps$ denote the standard coproduct and counit of $U(\mfg)$, respectively. 

For each $k>0$, set 
\begin{equation*}
\pi_k=\pi^{\otimes k}\circ \Delta^{(k-1)}: U(\mfg)\to \End(V) ^{\otimes k}\subset \End(V^{\otimes k}),
\end{equation*}
where $\Delta^{(k)}:U(\mfg)\to U(\mfg)^{\otimes (k+1)}$ is defined recursively by $\Delta^{(0)}=\mathrm{id}$ and 
\begin{equation*}
 \Delta^{(k)}=(\mathrm{id}^{\otimes (k-1)} \otimes \Delta) \circ \Delta^{(k-1)}. 
\end{equation*}
By coassociativity, how this is defined is not important. 
By convention, $\pi_0$ is the counit $\veps$. We then define 
\begin{equation*}
 \Psi:U(\mfg)\to \prod_{k\in \mathbb{N}}\End(V^{\otimes k}),\quad \mathrm{pr}_n\circ \Psi=\pi_n \quad \forall \; n\in \mathbb{N}.
\end{equation*}
Here $\mathbb{N}$ is the set of non-negative integers and $\mathrm{pr}_n:\prod_{k\in \mathbb{N}}\End(V^{\otimes k})\to \End(V^{\otimes n})$ is the $n$-th projection homomorphism. 
\begin{theorem}\label{T:PBW2}
 $\Psi$ is an injective homomorphism of algebras. 
\end{theorem}
We first prove the theorem in \S\ref{ssec:A} given the following assumption: 
\begin{enumerate}[label=(\Alph*),font=\upshape]
 \item \label{A}$\mathrm{id}_V\notin \pi(\mfg)$. 
\end{enumerate}
In particular, this holds when $\mfg$ has a trivial center. We will then explain in \S\ref{ssec:NoA} how to generalize to the case where $\mathrm{id}_V\in \pi(\mfg)$.

Let us set $\mfg_\pi=\pi(\mfg)$; by the faithfulness of $\pi$, this is a Lie subalgebra of $\mathfrak{gl}(V)$ isomorphic to $\mfg$. We also let $\{\mbF_k\}_{k\in \mathbb{N}}$ be the standard filtration on $U(\mfg)$ (so that $\mathrm{gr}(U(\mfg))\cong S(\mfg)$). 

\subsection{Proof of Theorem \ref{T:PBW2} given \ref{A}}\label{ssec:A}

Since $U(\mfg)=\bigcup_{k\in \mathbb{N}}\mbF_k$, it suffices to show that $\Psi|_{\mbF_k}$ is injective for each $k\in \mathbb{N}$. We will in fact prove the stronger assertion of the following lemma: 
\begin{lemma}\label{L:gen}
 For each $k\in \mathbb{N}$, 
 $
  \pi_k|_{\mbF_k}:\mbF_k\to \End(V)^{\otimes k}
 $
 is injective.
\end{lemma}
\begin{proof}
 The case $k=0$ is trivial, so let us fix $k\geq 1$.  It suffices to prove that
 \begin{equation}\label{suffices}
 \mathrm{Ker}(\pi_k)\cap (\mbF_\ell \setminus \mbF_{\ell-1})=\emptyset \quad \forall \quad 1\leq \ell\leq k. 
 \end{equation}
 Fix any such $\ell$ and define $S^\ell(\mfg_\pi)_k$ by 
 \begin{equation*}
  S^\ell(\mfg_\pi)_k=\iota_{\ell,k}(S^\ell(\mfg_\pi)),
 \end{equation*}
 where $\iota_{\ell,k}$ is the embedding 
 \begin{equation*}
  \iota_{\ell,k}:\mfg_\pi^{\otimes \ell}\to \mfg_\pi^{\otimes \ell}\otimes (\C\cdot\mathrm{id}_V)^{\otimes (k-\ell)} \subset \End(V)^{\otimes k}, \quad X\mapsto X\otimes \mathrm{id}_V^{\otimes (k-\ell)}.
 \end{equation*}
Let $\mcE_{\ell,k}\subset \End(V)^{\otimes k}$ be given by 
\begin{equation*}
 \mcE_{\ell,k}=\mathrm{Span}_\C\{y_1\otimes \cdots \otimes y_k\,:\, y_i\in \End(V), \; y_a=\mathrm{id}_V\; \text{ for some }\; 1\leq a\leq \ell\}.
\end{equation*}
This space satisfies $\mcE_{\ell,k}\cap \iota_{\ell,k}(\mfg^{\otimes \ell})=\{0\}$ and thus $\mcE_{\ell,k}\cap S^\ell(\mfg_\pi)_k=\{0\}$, as can be seen by extending any basis of $\mfg_\pi$ to a basis of $\End(V)$ containing $\mathrm{id}_V$ and applying the assumption \ref{A}. We may therefore choose a linear projection
\begin{equation*}
 \mathds{P}_{\ell,k}:\End(V)^{\otimes k}\onto S^\ell(\mfg_\pi)_k \quad \text{ with } \quad \mathds{P}_{\ell,k}|_{\mcE_{\ell,k}}=0. 
\end{equation*}
 Consider the composite
\begin{equation*}
 \omega_{\ell,k}=\mathds{P}_{\ell,k}\circ \pi_k|_{\mbF_\ell}: \mbF_\ell\to S^\ell(\mfg_\pi)_k.
\end{equation*}

\noindent \textit{Claim:} $\omega_{\ell,k}(\mbF_{\ell-1})=0$. 

\medskip 

Consider a product $x_1\cdots x_m$ with $x_i\in \mfg$ and $m\leq \ell-1$.
Since 
\begin{equation*}
\Delta^{(k-1)}(x_i)=\sum_{a=1}^k (x_i)_a,\quad \text{ where }\quad (x)_a=1^{\otimes (a-1)}\otimes x \otimes 1^{\otimes (k-a)},
\end{equation*}
we have 
\begin{equation}\label{pik-mon}
 \pi_k(x_1\cdots x_m)=\sum_{a_1,\ldots,a_m=1}^k (\pi(x_1))_{a_1}\cdots (\pi(x_m))_{a_m}\in \mcE_{\ell,k}. 
\end{equation}
Applying $\mathds{P}_{\ell,k}$ then gives $\omega_{\ell,k}(x_1\cdots x_m)=0$, from which the claim follows.

Consequently, $\omega_{\ell,k}$ induces a linear map
\begin{equation*}
 \varpi_{\ell,k}:\mbF_\ell/\mbF_{\ell-1}\to S^\ell(\mfg_\pi)_k.
\end{equation*}
To complete the proof of \eqref{suffices}, it is enough to show that $ \varpi_{\ell,k}$ is injective. In fact, it is an isomorphism. 

\medskip 

\noindent \textit{Claim}: $\varpi_{\ell,k}$ is an isomorphism. 

\medskip

This is essentially just the PBW theorem for $U(\mfg)$. Using the formula \eqref{pik-mon} with $m=\ell$, we find that
\begin{align*}
 \pi_k(x_1\cdots x_\ell)&\equiv  \sum_{\substack{1\leq a_i\leq \ell\\ a_i\neq a_j \; \forall \,i\neq j}} (\pi(x_1))_{a_1}\cdots (\pi(x_\ell))_{a_\ell} \mod \mcE_{\ell,k}\\
                     &\equiv \sum_{\sigma\in S_\ell}\pi(x_{\sigma(1)})\otimes \cdots \otimes \pi(x_{\sigma(\ell)}) \mod \mcE_{\ell,k}
\end{align*}
It follows that 
\begin{equation*}
 \varpi_{\ell,k}(\bar x_1 \cdots \bar x_\ell)=\sum_{\sigma\in S_\ell}\pi(x_{\sigma(1)})\otimes \cdots \otimes \pi(x_{\sigma(\ell)}).
\end{equation*}
By the PBW Theorem for $U(\mfg)$, we already know $\mbF_\ell/\mbF_{\ell-1}\cong S^\ell(\mfg)$ and, after viewing $S^\ell(\mfg)$ as the subspace of $\mfg^{\otimes \ell}$ consisting of symmetric tensors (as we have been doing above), the standard identification is given by the symmetrizing map
\begin{equation*}
 \bar x_1 \cdots \bar x_\ell \mapsto \frac{1}{\ell!}\sum_{\sigma \in S_\ell} x_{\sigma(1)}\otimes \cdots \otimes x_{\sigma(\ell)}.
\end{equation*}
After identifying $\mfg$ with $\mfg_\pi$ and renormalizing, this is precisely $\varpi_{\ell,k}$. Hence $\varpi_{\ell,k}$ is an isomorphism.
\end{proof}

\subsection{Proof of Theorem \ref{T:PBW2} in general}\label{ssec:NoA}

We now consider the case where $\mathrm{id}_V\in \pi(\mfg)$. In this case $\mfg$ has a non-trivial central element $\mathds{1}$ such that 
\begin{equation*}
 \pi(\mathds{1})=\mathrm{id}_V. 
\end{equation*}
Choose a subspace $\mfa$ of $\mfg$ complimentary to $\C\mathds{1}$ (not necessarily a Lie subalgebra of $\mfg$). Let us fix an ordered basis 
$\{x_\lambda\}_{\lambda\in \Lambda}$ of $\mfa$, and let $\mbU$ denote the subspace of $U(\mfg)$ spanned by ordered monomials in this basis. In particular, we have the vector space decomposition 
\begin{equation*}
 U(\mfg)\cong \C[\mathds{1}]\otimes \mbU. 
\end{equation*}

The standard filtration $\mbF_k$ on $U(\mfg)$ induces a filtration $\{\mbF_k^\mfa\}_{k\in \mathbb{N}}$ on $\mbU$ given by $\mbF_k^\mfa = \mbF_k \cap \mbU$. Let $\pi_k^\mfa=\pi_k|_{\mbU}$. 

\medskip 

\noindent \textit{Claim:} $\pi_k^\mfa|_{\mbF_k^\mfa}:\mbF_k^\mfa\to \End(V)^{\otimes k}$ is injective for each $k\in \mathbb{N}$. 

\medskip

The claim does not automatically follow from Lemma \ref{L:gen} since $\mfa$ may not be a Lie subalgebra of $\mfg$ and hence $\pi_k^\mfa$ is no longer a Lie algebra representation. However, the proof of Lemma \ref{L:gen} does still go through in our present setting; one just needs to know that the symmetrization map still provides an isomorphism  
\begin{equation*}
\mbF_k^\mfa/\mbF_{k-1}^\mfa\stackrel{\sim}{\rightarrow} S^k(\mfa) \quad \forall \quad k\geq 1, 
\end{equation*}
which is a consequence of the Poincar\'{e}-Birkhoff-Witt Theorem for $\mfg$.

Now let $X\in U(\mfg)$ be an arbitrary element. Choose $\ell\in \mathbb{N}$ such that $X\in \mbF_\ell$. Then $X=P(\mathds{1})$, where $P(\mathds{1})$ is a polynomial in $\mathds{1}$ with coefficients in $\mbU$ of degree at most $\ell$. Since $\pi_k(\mathds{1})=k\cdot \mathrm{id}_{V^{\otimes k}}$ for each $k$, we have   
\begin{equation*}
 \pi_k(X)=\pi_k^\mfa(P(k)) \quad \forall \;k\in \mathbb{N}. 
\end{equation*}
It follows that if $\Psi(X)=0$, then $P(k)\in \mathrm{Ker}(\pi_k^\mfa)$ for all $k\in \mathbb{N}$. Moreover, by assumption, $P(k)$ belongs to $\mbF_\ell^\mfa$ for each $k$. By the above claim, this means that $P(k)=0$ for all $k\geq \ell$. This is only possible if $X=P(\mathds{1})=0$. $\hfill \square$


\end{document}